\documentclass[]{article}
\usepackage[english]{babel}

\usepackage{amsmath}
\usepackage{amssymb}
\usepackage{amsthm}       

\usepackage[margin=1.5in]{geometry} 

\newtheorem{thm}{Theorem}[section]
\newtheorem{prop}[thm]{Proposition}

\newtheorem{exam}[thm]{Example}

\newtheorem{lem}[thm]{Lemma}

\begin{document}

\title{Real analytic families of harmonic functions in a domain with a small hole}
\author{M.~Dalla Riva \footnote{The research  of M.~Dalla Riva was supported by {\it FEDER} funds through {\it COMPETE}--Operational Programme Factors of Competitiveness (``Programa Operacional Factores de Competitividade'') and by Portuguese funds through the {\it Center for Research and Development in Mathematics and Applications} (University of Aveiro) and the Portuguese Foundation for Science and Technology (``FCT--Funda{\c c}{\~a}o para a Ci\^encia e a Tecnologia''), within project PEst-C/MAT/UI4106/2011 with the COMPETE  number FCOMP-01-0124-FEDER-022690. The research was also supported  by the Portuguese Foundation for Science and Technology (``FCT--Funda{\c c}{\~a}o para a Ci\^encia e a Tecnologia'')  with the research grant SFRH/BPD/64437/2009.}
  \&  P.~Musolino}

\date{}
\maketitle

\noindent
{\bf Abstract:} Let $n\ge 3$. Let $\Omega^i$ and $\Omega^o$ be open bounded connected subsets of $\mathbb{R}^n$ containing the origin.  Let $\epsilon_0>0$ be such that $\Omega^o$ contains the closure of $\epsilon\Omega^i$ for all $\epsilon\in]-\epsilon_0,\epsilon_0[$. Then, for a fixed $\epsilon\in]-\epsilon_0,\epsilon_0[\setminus\{0\}$ we consider a Dirichlet problem for the Laplace operator in the perforated domain $\Omega^o\setminus\epsilon\Omega^i$. We denote by $u_\epsilon$ the corresponding solution. If $p\in\Omega^o$ and $p\neq 0$, then we know that under suitable regularity assumptions there exist $\epsilon_p>0$ and a real analytic operator $U_p$ from $]-\epsilon_p,\epsilon_p[$ to $\mathbb{R}$ such that $u_\epsilon(p)=U_p[\epsilon]$ for all $\epsilon\in]0,\epsilon_p[$. Thus it is natural to ask what happens to the equality $u_\epsilon(p)=U_p[\epsilon]$ for $\epsilon$ negative. We show a general result on continuation properties of some particular real analytic families of harmonic functions in domains with a small hole and we prove that the validity of the equality $u_\epsilon(p)=U_p[\epsilon]$ for $\epsilon$ negative depends on the parity of the dimension $n$.

\vspace{11pt}

\noindent
{\bf Keywords:} singularly perturbed perforated domains, harmonic functions, real analytic continuation in Banach space. 

\noindent
{\bf MSC 2010:} 31B05, 31B10, 35B25, 35C20, 35J25.

\section{Introduction}\label{intro}
We fix once for all
\[
n\in\mathbb{N}\,,\ n\ge 3\,,\qquad\alpha\in]0,1[\,.
\] Here $\mathbb{N}$ denotes the set of natural numbers including $0$. Then we fix two sets $\Omega^i$ and $\Omega^o$ in the $n$-dimensional Euclidean space $\mathbb{R}^n$. The letter `$i$' stands for `inner domain' and the letter `$o$' stands for `outer domain'. We assume that $\Omega^i$ and $\Omega^o$ satisfy the following condition
\begin{eqnarray}\label{e1}
&&\Omega^i\ \mathrm{and}\ \Omega^o\ \mathrm{are\ open\ bounded\ connected\ subsets\ of}\ \mathbb{R}^n\ \mathrm{of}\\
\nonumber
&&\mathrm{class}\ C^{1,\alpha}\ \mathrm{such\ that\ }\mathbb{R}^n\setminus\mathrm{cl}\Omega^i\,\mathrm{and\ }\mathbb{R}^n\setminus\mathrm{cl}\Omega^o\ \mathrm{are\ connected, }\\
\nonumber
&&\mathrm{and\ such\ that\ the\ origin}\ 0\ \mathrm{of}\ \mathbb{R}^n\ \mathrm{belongs\ both\ to}\ \Omega^i\ \mathrm{and}\ \Omega^o\,.
\end{eqnarray} Here $\mathrm{cl}\Omega$ denotes the closure of $\Omega$ for all $\Omega\subseteq\mathbb{R}^n$. For the definition of functions and sets of the usual Schauder class $C^{0,\alpha}$ and $C^{1,\alpha}$, we refer for example to Gilbarg and Trudinger~\cite[\S6.2]{GiTr01}.  We note that condition \eqref{e1} implies that $\Omega^i$ and $\Omega^o$ have no holes and that there exists a real number $\epsilon_0$ such that
\begin{equation}\label{e2}
\epsilon_0>0\ \mathrm{and\ }\ \epsilon\mathrm{cl}\Omega^i\subseteq \Omega^o\ \mathrm{for\ all}\ \epsilon\in]-\epsilon_0,\epsilon_0[\,.
\end{equation}  Then we denote by $\Omega(\epsilon)$ the perforated domain defined by
\[
\Omega(\epsilon)\equiv\Omega^o\setminus(\epsilon\mathrm{cl}\Omega^i)\qquad\quad\forall\epsilon\in]-\epsilon_0,\epsilon_0[\,.
\]
A simple topological argument shows that 
$\Omega(\epsilon)$ is an open bounded connected subset of $\mathbb{R}^n$ of class $C^{1,\alpha}$ for all $\epsilon\in]-\epsilon_0,\epsilon_{0}[\setminus\{0\}$. Moreover,  the boundary  
$\partial \Omega(\epsilon)$ of $\Omega(\epsilon)$ has exactly the two connected components $\partial \Omega^o$ and $\epsilon\partial\Omega^i$, for all 
$\epsilon\in]-\epsilon_0,\epsilon_{0}[$. We also note that $\Omega(0)=\Omega^o\setminus\{0\}$.\par
Now  let  $f^i\in C^{1,\alpha}(\partial\Omega^i)$ and $f^o\in C^{1,\alpha}(\partial\Omega^o)$. Let $\epsilon\in]-\epsilon_0,\epsilon_{0}[\setminus\{0\}$. We consider the following boundary value problem
\begin{equation}\label{dir1}
\left\{
\begin{array}{ll}
\Delta u=0&\text{ in }\Omega(\epsilon)\,,\\
u(x)=f^i(x/\epsilon)&\text{ for }x\in\epsilon\partial\Omega^i\,,\\
u(x)=f^o(x)&\text{ for }x\in\partial\Omega^o\,.
\end{array}
\right.
\end{equation}
As is well known, the problem in \eqref{dir1} has a unique solution in $C^{1,\alpha}(\mathrm{cl}\Omega(\epsilon))$. We denote such a solution by $u_\epsilon$.  Then we fix a point $p$ in $\Omega^o\setminus\{0\}$ and we take $\epsilon_p\in]0,\epsilon_0[$ such that $p\in\Omega(\epsilon)$ for all $\epsilon\in]0,\epsilon_p[$. In particular, it makes sense to consider $u_\epsilon(p)$ for all $\epsilon\in]0,\epsilon_p[$. Thus we can ask the following question.
\[
\text{What can be said of the map from $]0,\epsilon_p[$ to $\mathbb{R}$ which takes $\epsilon$ to $u_\epsilon(p)$?}
\] Questions of this type have been largely investigated by the so called Asymptotic Analysis. We mention here as an example the work of Maz'ya, Nazarov, and Plamenevskij in \cite{MaNaPl00}. The techniques of Asymptotic Analysis aim at representing the behavior of $u_\epsilon(p)$ as $\epsilon\to 0^+$ in terms of regular functions of $\epsilon$ plus a remainder which is smaller than a known infinitesimal function of $\epsilon$.  Instead, by the different approach proposed by Lanza de Cristoforis (cf.~\textit{e.g.}, Lanza de Cristoforis \cite{La02}) and by possibly shrinking $\epsilon_p$, we can represent the function which takes $\epsilon$ to $u_\epsilon(p)$ as the restriction to $]0,\epsilon_p[$ of a real analytic map defined on $]-\epsilon_p,\epsilon_p[$ (for the definition and properties of real analytic maps in Banach space we refer, {\it e.g.}, to Deimling~\cite[\S15]{De85}.)  Moreover, we can consider what we call the `macroscopic' behaviour of the family $\{u_\epsilon\}_{\epsilon\in]0,\epsilon_0[}$. Indeed, if $\Omega_M\subseteq\Omega^o$ is open, and  $0\notin\mathrm{cl}\Omega_M$, and $\epsilon_M\in]0,\epsilon_0]$ is such that $\mathrm{cl}\Omega_M\cap(\epsilon\mathrm{cl}\Omega^i)=\emptyset$ for all $\epsilon\in]-\epsilon_M,\epsilon_M[$, then $\mathrm{cl}\Omega_M\subseteq\mathrm{cl}\Omega(\epsilon)$ for all $\epsilon\in]0,\epsilon_M[$. Thus it makes sense to consider the restriction $u_{\epsilon|\mathrm{cl}\Omega_M}$ for all $\epsilon\in]0,\epsilon_M[$. In particular, it makes sense to consider the map from $]0,\epsilon_M[$ to $C^{1,\alpha}(\mathrm{cl}\Omega_M)$ which takes $\epsilon$ to $u_{\epsilon|\mathrm{cl}\Omega_M}$.  Then we prove in  Proposition~\ref{right} that there exists a real number $\epsilon_1\in]0,\epsilon_0]$ such that the following statement holds (see also  Lanza de Cristoforis~\cite[Thm.~5.3]{La08}.)
\begin{itemize}
\item[(a1)] Let $\Omega_M\subseteq\Omega^o$ be open and such that $0\notin\mathrm{cl}\Omega_M$. Let $\epsilon_M\in]0,\epsilon_1]$ be such that $\mathrm{cl}\Omega_M\cap(\epsilon\mathrm{cl}\Omega^i)=\emptyset$ for all $\epsilon\in]-\epsilon_M,\epsilon_M[$. Then there exists a real analytic operator $U_M$ from $]-\epsilon_M,\epsilon_M[$ to $C^{1,\alpha}(\mathrm{cl}\Omega_M)$ such that 
\begin{equation}\label{uUM}
u_{\epsilon|\mathrm{cl}\Omega_M}=U_M[\epsilon]\qquad\forall\epsilon\in]0,\epsilon_M[\,.
\end{equation}
\end{itemize} Here the letter `$M$' stands for `macroscopic'.  But we can also consider the `microscopic' behavior of the family $\{u_\epsilon\}_{\epsilon\in]0,\epsilon_0[}$ in proximity of the boundary of the hole. To do so we denote by $u_\epsilon(\epsilon\,\cdot\,)$ the rescaled function which takes $x\in(1/\epsilon)\mathrm{cl}\Omega(\epsilon)$ to $u_\epsilon(\epsilon x)$, for all $\epsilon\in]0,\epsilon_0[$. If $\Omega_m\subseteq\mathbb{R}^n\setminus\mathrm{cl}\Omega^i$ is open, and $\epsilon_m\in]0,\epsilon_0]$ is such that $\epsilon\mathrm{cl}\Omega_m\subseteq\Omega^o$ for all $\epsilon\in]-\epsilon_m,\epsilon_m[$, then $\mathrm{cl}\Omega_m\subseteq(1/\epsilon)\mathrm{cl}\Omega(\epsilon)$ for all $\epsilon\in]0,\epsilon_m[$ and it makes sense to consider the map from $]0,\epsilon_m[$ to $C^{1,\alpha}(\mathrm{cl}\Omega_m)$ which takes $\epsilon$ to $u_\epsilon(\epsilon\,\cdot\,)_{|\mathrm{cl}\Omega_m}$.  In  Proposition~\ref{right} we prove that there exists $\epsilon_1\in]0,\epsilon_0]$ such that the following statement holds.
\begin{itemize}
\item[(a2)] Let $\Omega_m\subseteq \mathbb{R}^n\setminus\mathrm{cl}\Omega^i$ be open and bounded. Let $\epsilon_m\in]0,\epsilon_1]$ be such that $\epsilon\mathrm{cl}\Omega_m\subseteq\Omega^o$ for all $\epsilon\in]-\epsilon_m,\epsilon_m[$. Then there exists a real analytic operator $U_m$ from $]-\epsilon_m,\epsilon_m[$ to $C^{1,\alpha}(\mathrm{cl}\Omega_m)$ such that
\begin{equation}\label{uUm}
u_{\epsilon}(\epsilon\,\cdot\,)_{|\mathrm{cl}\Omega_m}=U_m[\epsilon]\qquad\forall\epsilon\in]0,\epsilon_m[\,. 
\end{equation}
\end{itemize}   Here the letter `$m$' stands for `microscopic'.  \par
We now observe that Proposition~\ref{right} states that the equalities in \eqref{uUM} and \eqref{uUm} hold in general only for $\epsilon$ positive, but the functions $u_{\epsilon|\mathrm{cl}\Omega_M}$, $U_M[\epsilon]$, $u_\epsilon(\epsilon\,\cdot\,)_{|\mathrm{cl}\Omega_m}$  and $U_m[\epsilon]$ are defined also for $\epsilon$ negative. Thus, it is natural to formulate the following question.
\begin{equation}\label{?}
\text{What happens to the equalities in \eqref{uUM} and \eqref{uUm} for $\epsilon$ negative?}
\end{equation} 

The purpose of this paper is to answer to the question formulated here above. In particular, 
we prove in Theorem~\ref{even} that the equalities in \eqref{uUM} and \eqref{uUm} hold also for $\epsilon$ negative if the dimension $n$ is even. Instead, if the dimension $n$ is odd  we show in Proposition~\ref{const} that the equalities in \eqref{uUM} and \eqref{uUm}  hold for $\epsilon$ negative only if there exists a real constant $c$ such that $f^i=c$ and $f^o=c$ identically (so that $u_\epsilon(x)=c$ for all $x\in\mathrm{cl}\Omega(\epsilon)$ and $\epsilon\in]-\epsilon_0,\epsilon_0[\setminus\{0\}$.)\par However, we note that the conditions expressed in (a1) and (a2) are not related to the particular boundary value problem in \eqref{dir1}. Indeed, we could prove the validity of (a1) and (a2) for families of functions $\{u_\epsilon\}_{\epsilon\in]0,\epsilon_1[}$ which are solutions of problems with different boundary conditions, such as those considered in Lanza de Cristoforis~\cite{ La08, Lan07a, La10}. For this reason, we investigate the properties of families of functions $\{u_\epsilon\}_{\epsilon\in]0,\epsilon_1[}$ such that
\begin{itemize}
\item[(a0)] $u_\epsilon\in C^{1,\alpha}(\mathrm{cl}\Omega(\epsilon))$ and $\Delta u_{\epsilon}=0$ in $\Omega(\epsilon)$ for all $\epsilon\in]0,\epsilon_1[$
\end{itemize} and which satisfy the conditions in (a1) and (a2), but which are not required to satisfy any specific boundary condition on $\partial\Omega(\epsilon)$. To do so, we introduce the following terminology.\par
Let $\epsilon_1\in]0,\epsilon_0]$. We say that $\{u_\epsilon\}_{\epsilon\in]0,\epsilon_1[}$ is a {\em right real analytic family of harmonic functions on $\Omega(\epsilon)$}  if it satisfies the conditions in (a0), (a1), (a2). We say that $\{v_\epsilon\}_{\epsilon\in]-\epsilon_1,\epsilon_1[}$ is a  {\em real analytic family of harmonic functions on $\Omega(\epsilon)$} if it satisfies the following conditions (b0)--(b2).
\begin{enumerate}
\item[(b0)] $v_0\in  C^{1,\alpha}(\mathrm{cl}\Omega^o)$ and $\Delta v_0=0$ in $\Omega^o$, $v_\epsilon\in C^{1,\alpha}(\mathrm{cl}\Omega(\epsilon))$ and $\Delta v_{\epsilon}=0$ in $\Omega(\epsilon)$ for all $\epsilon\in]-\epsilon_1,\epsilon_1[\setminus\{0\}$.
\item[(b1)] Let $\Omega_M\subseteq\Omega^o$ be open and such that $0\notin\mathrm{cl}\Omega_M$. Let $\epsilon_M\in]0,\epsilon_1]$ be such that  $\mathrm{cl}\Omega_M\cap\epsilon\mathrm{cl}\Omega^i=\emptyset$ for all $\epsilon\in]-\epsilon_M,\epsilon_M[$. Then there exists a real analytic operator $V_M$ from $]-\epsilon_M,\epsilon_M[$ to $C^{1,\alpha}(\mathrm{cl}\Omega_M)$ such that
\begin{equation*}
v_{\epsilon|\mathrm{cl}\Omega_M}=V_M[\epsilon]\qquad\forall\epsilon\in]-\epsilon_M,\epsilon_M[\,. 
\end{equation*}
\item[(b2)] Let $\Omega_m\subseteq \mathbb{R}^n\setminus\mathrm{cl}\Omega^i$ be an open and bounded subset. Let $\epsilon_m\in]0,\epsilon_1]$ be such that $\epsilon\mathrm{cl}\Omega_m\subseteq\Omega^o$ for all $\epsilon\in]-\epsilon_m,\epsilon_m[$. Then there exists a real analytic operator $V_m$ from $]-\epsilon_m,\epsilon_m[$ to $C^{1,\alpha}(\mathrm{cl}\Omega_m)$ such that
\begin{equation}\label{vVm}
v_{\epsilon}(\epsilon\,\cdot\,)_{|\mathrm{cl}\Omega_m}=V_m[\epsilon]\qquad\forall\epsilon\in]-\epsilon_m,\epsilon_m[\setminus \{0\}\,. 
\end{equation}
\end{enumerate}
Here $v_\epsilon(\epsilon\,\cdot\,)$ denotes the map which takes $x\in(1/\epsilon)\mathrm{cl}\Omega(\epsilon)$ to $v_\epsilon(\epsilon x)$, for all $\epsilon\in]-\epsilon_1,\epsilon_1[\setminus\{0\}$. We also note that we do not ask in condition (b2) that the equality in \eqref{vVm} holds for $\epsilon=0$. In particular, $v_0(0\,\cdot\,)_{|\mathrm{cl}\Omega_m}$ is necessarily a constant function on $\mathrm{cl}\Omega_m$, while $V_m[0]$ may be non-constant. Finally, we say that  $\{w_\epsilon\}_{\epsilon\in]-\epsilon_1,\epsilon_1[}$ is a  {\em real analytic family of harmonic functions on $\Omega^o$}  if it satisfies the following conditions (c0), (c1).
\begin{enumerate}
\item[(c0)] $w_\epsilon\in C^{1,\alpha}(\mathrm{cl}\Omega^o)$ and $\Delta w_{\epsilon}=0$ in $\Omega^o$ for all $\epsilon\in]-\epsilon_1,\epsilon_1[$.
\item[(c1)] The map from $]-\epsilon_1,\epsilon_1[$ to $C^{1,\alpha}(\mathrm{cl}\Omega^o)$ which takes $\epsilon$ to $w_\epsilon$ is real analytic. 
\end{enumerate}

We state our main results in Theorems 3.1 and 3.2, where we consider separately the case of dimension  $n$ even and of dimension $n$  odd, respectively. In particular, by Theorems 3.1 and 3.2 we can deduce the validity of the following statements (j) and (jj).
\begin{enumerate}
\item[(j)] If the dimension $n$ is even and  $\{u_\epsilon\}_{\epsilon\in]0,\epsilon_1[}$ is a {\em right real analytic family of harmonic functions on} $\Omega(\epsilon)$, then there exists a {\em real analytic family of harmonic functions} $\{v_\epsilon\}_{\epsilon\in]-\epsilon_1,\epsilon_1[}$ on $\Omega(\epsilon)$ such that $u_\epsilon=v_\epsilon$ for all $\epsilon\in]0,\epsilon_1[$.
\item[(jj)] If the dimension $n$ is odd and   $\{v_\epsilon\}_{\epsilon\in]-\epsilon_1,\epsilon_1[}$ is a {\em real analytic family of harmonic functions on} $\Omega(\epsilon)$, then there exists a {\em real analytic family of harmonic functions} $\{w_\epsilon\}_{\epsilon\in]-\epsilon_1,\epsilon_1[}$ {\em on} $\Omega^o$ such that $v_\epsilon=w_{\epsilon|\mathrm{cl}\Omega(\epsilon)}$ for all $\epsilon\in]-\epsilon_1,\epsilon_1[$.
\end{enumerate}
In particular we note that for $n$ odd statement (jj) implies that for each $\epsilon \in ]-\epsilon_1,\epsilon_1[$ the function $v_\epsilon$ can be extended inside the hole $\epsilon \Omega^i$ to an harmonic function defined on the whole of $\Omega^o$. As is well known, the condition of existence of an extension of a
harmonic function defined on
$\Omega(\epsilon)$ to $\Omega$ is quite restrictive. Hence, case (jj)
has to be considered, in a sense, as exceptional. 
\par 
The paper is organized as follows. Section~\ref{pre} is a section of preliminaries where we introduce some known results of Potential Theory. In particular, we adopt the approach proposed by Lanza de Cristoforis for the analysis of elliptic boundary value problems in domains with a small hole. Accordingly, we show that the boundary value problem in \eqref{dir1} is equivalent to a suitable functional equation $\Lambda=0$, where $\Lambda$ is a real analytic operator between Banach spaces.   Then we analyze equation $\Lambda=0$ by exploiting the Implicit Function Theorem for real analytic functions (cf.~{\it e.g.}, Deimling~\cite[Theorem~15.3]{De85}.)  In Section~\ref{main} we prove our main Theorems~\ref{even} and \ref{odd}, where we consider separately case $n$ even and $n$ odd, respectively. Then in Examples~\ref{ex1}, \ref{ex2} and \ref{ex3} we show that the the assumptions in Theorems~\ref{even} and \ref{odd} cannot be weakened in a sense which we clarify below. In particular, by Examples~\ref{ex2} and \ref{ex3} we deduce that analogs of statements (j) and (jj) do not hold if we replace the assumption that $u_\epsilon$, $v_\epsilon$, $w_\epsilon$ are harmonic with the weaker assumption that $u_\epsilon$, $v_\epsilon$, $w_\epsilon$  are real analytic.    In the last Section~\ref{appl} we consider some particular cases and we show some applications of Theorems~\ref{even} and \ref{odd}.  In Proposition~\ref{right} we consider the family $\{u_\epsilon\}_{\epsilon\in]0,\epsilon_0[}$ of the solutions in $C^{1,\alpha}(\mathrm{cl}\Omega(\epsilon))$ of \eqref{dir1}. We show that there exists $\epsilon_1\in]0,\epsilon_0]$ such that $\{u_\epsilon\}_{\epsilon\in]0,\epsilon_0[}$ satisfies the conditions in (a1) and (a2).  We also prove that we can take $\epsilon_1=\epsilon_0$ if the dimension $n$ is even.  In Proposition~\ref{sym} we assume that $n$ is even  and we consider a right real analytic  family $\{u_\epsilon\}_{\epsilon\in]0,\epsilon_1[}$ of harmonic function  on $\Omega(\epsilon)$. Then, conditions (a1) and (a2) imply that  $u_{\epsilon|\mathrm{cl}\Omega_M}$ and $u_{\epsilon}(\epsilon\,\cdot\,)_{|\mathrm{cl}\Omega_m}$ can be represented  by means of convergent power series of $\epsilon$ for $\epsilon$ small and positive. Under the condition that either $\Omega^i=-\Omega^i$ or $\Omega^o=-\Omega^o$ and that $\{u_\epsilon\}_{\epsilon\in]0,\epsilon_1[}$ satisfies some suitable symmetry assumptions, we obtain some additional information on the power series expansion of  $u_{\epsilon|\mathrm{cl}\Omega_M}$ and $u_{\epsilon}(\epsilon\,\cdot\,)_{|\mathrm{cl}\Omega_m}$ for $\epsilon$ small and positive. Finally, in Proposition~\ref{const} we  assume that $n$ is odd and we answer to the question in \eqref{?} by exploiting Theorem~\ref{odd}.

\section{Preliminaries}\label{pre}

We denote by $S_n$ the function from $\mathbb{R}^n\setminus\{0\}$ to $\mathbb{R}$ defined by
\[
S_n(x)\equiv\frac{|x|^{2-n}}{(2-n)s_n}\qquad\forall x\in\mathbb{R}^n\setminus\{0\}\,.
\] Here $s_n$ denotes the $(n-1)$ dimensional measure of the unit sphere in $\mathbb{R}^n$. As is well known $S_n$ is the fundamental solution of the Laplace operator in $\mathbb{R}^n$.
Let $\Omega$ be an open bounded subset of $\mathbb{R}^n$ of class $C^{1,\alpha}$. Let $\mu\in C^{0,\alpha}(\partial\Omega)$. Then we denote by $v[\mu]$ the single layer potential of density $\mu$. Namely $v[\mu]$ is the function from $\mathbb{R}^n$ to $\mathbb{R}$ defined by
\[
v[\mu](x)\equiv\int_{\partial\Omega}S_n(x-y)\mu(y)\, d\sigma_y\qquad\forall x\in\mathbb{R}^n\,.
\]

Then we have the following well known Lemma, whose proof is based on classical results of Potential Theory (see also Miranda \cite[Theorem~5.I]{Mi65}.)

\begin{lem}\label{homeo}
Let $\Omega$ be an open bounded subset of $\mathbb{R}^n$ of class $C^{1,\alpha}$. Let $\tilde\Omega$ be an open bounded subset of $\mathbb{R}^n\setminus\mathrm{cl}\Omega$. Then the map from $C^{0,\alpha}(\partial \Omega)$ to $C^{1,\alpha}(\mathrm{cl}\Omega)$ which takes $\mu$ to $v[\mu]_{|\mathrm{cl}\Omega}$ is linear and continuous, and the map from $C^{0,\alpha}(\partial \Omega)$ to $C^{1,\alpha}(\mathrm{cl}\tilde\Omega)$ which takes $\mu$ to $v[\mu]_{|\mathrm{cl}\tilde\Omega}$ is linear and continuous. Moreover, the map from $C^{0,\alpha}(\partial\Omega)$ to $C^{1,\alpha}(\partial\Omega)$ which takes $\mu$ to $v[\mu]_{|\partial\Omega}$ is a linear homeomorphism. 
\end{lem}

We observe that the last sentence of Lemma~\ref{homeo} holds only if the dimension $n$ is greater or equal than $3$.
Indeed, in the planar case the map which takes $\mu$ to $v[\mu]_{|\partial\Omega}$ is not in general an homeomorphism from $C^{0,\alpha}(\partial\Omega)$ to  $C^{1,\alpha}(\partial\Omega)$ (see {\it e.g.}~Lanza de Cristoforis~\cite[pp.~949--950]{Lan07a}.)  In this paper have assumed that $n\ge 3$ and  thus we can exploit  Lemma~\ref{homeo} to convert a Dirichlet boundary value problem for the Laplace operator into a system of integral equations. In order to study the integral equations corresponding to the Dirichlet problem in the perforated domain $\Omega(\epsilon)$, with $\epsilon\in]-\epsilon_0,\epsilon_0[\setminus\{0\}$, we  now introduce the operators $\Lambda_1$ and $\Lambda_{-1}$. Let $\theta \in \{-1,1\}$. Then we denote by $\Lambda_\theta\equiv(\Lambda^i_\theta,\Lambda^o_\theta)$ the operator from $]-\epsilon_0,\epsilon_0[\times C^{1,\alpha}(\partial\Omega^i)\times C^{1,\alpha}(\partial\Omega^o)\times C^{0,\alpha}(\partial\Omega^i)\times C^{0,\alpha}(\partial\Omega^o)$ to $C^{1,\alpha}(\partial\Omega^i)\times C^{1,\alpha}(\partial\Omega^o)$ defined by 
\begin{eqnarray*}
\lefteqn{\Lambda^i_\theta[\epsilon,f^i,f^o,\mu^i,\mu^o](x)\equiv \theta\int_{\partial\Omega^i}S_n(x - y)\mu^i(y)\,d\sigma_y}\\
&&
\qquad\qquad\qquad\qquad
+\int_{\partial\Omega^o}S_n(\epsilon x - y)\mu^o(y)\,d\sigma_y-f^i(x)\quad\forall x\in\partial\Omega^i\,,\\
\lefteqn{\Lambda^o_\theta[\epsilon,f^i,f^o,\mu^i,\mu^o](x)\equiv\epsilon^{n-2}\int_{\partial\Omega^i}S_n(x - \epsilon y)\mu^i(y)\,d\sigma_y}\\
&&
\qquad\qquad\qquad\qquad
+\int_{\partial\Omega^o}S_n(x - y)\mu^o(y)\,d\sigma_y-f^o(x)\quad\forall x\in\partial\Omega^o
\end{eqnarray*} for all $(\epsilon,f^i,f^o,\mu^i,\mu^o)\in ]-\epsilon_0,\epsilon_0[\times C^{1,\alpha}(\partial\Omega^i)\times C^{1,\alpha}(\partial\Omega^o)\times C^{0,\alpha}(\partial\Omega^i)\times C^{0,\alpha}(\partial\Omega^o)$. Then, by Lemma~\ref{homeo} we deduce the validity of the following Proposition~\ref{L}.

\begin{prop}\label{L}
Let $\Omega^i$, $\Omega^o$ be as in \eqref{e1}. Let $\epsilon_0$ be as in \eqref{e2}. Let $\epsilon\in]-\epsilon_0,\epsilon_0[\setminus\{0\}$.  Let $(f^i, f^o)\in C^{1,\alpha}(\partial\Omega^i)\times C^{1,\alpha}(\partial\Omega^o)$.  Let $\theta\equiv(\mathrm{sgn}\, \epsilon)^n$. Then there exists a unique pair of functions $(\mu^i, \mu^o)\in C^{0,\alpha}(\partial\Omega^i)\times C^{0,\alpha}(\partial\Omega^o)$ such that 
\begin{equation}\label{L=0}
\Lambda_\theta[\epsilon,f^i,f^o,\mu^i,\mu^o]=(0,0)\,.
\end{equation} Moreover, the function $u$ from $\mathrm{cl}\Omega(\epsilon)$ to $\mathbb{R}$ defined by
\[
u(x)\equiv \epsilon^{n-2}\int_{\partial\Omega^i}S_n(x - \epsilon y)\mu^i(y)\,d\sigma_y+\int_{\partial\Omega^o}S_n(x - y)\mu^o(y)\,d\sigma_y\quad \forall x\in\mathrm{cl}\Omega(\epsilon)
\] is the unique solution in $C^{1,\alpha}(\mathrm{cl}\Omega(\epsilon))$ of the boundary value problem in \eqref{dir1}. 
\end{prop}
\begin{proof} It is a straightforward consequence of the Theorem of change of variables in integrals, of well known properties of functions in Schauder spaces, and of Lemma~\ref{homeo}.  
\end{proof}

\medskip

We note that the system of equations in \eqref{dir1} is defined for $\epsilon\neq 0$. Instead we can consider equation $\Lambda_{\theta}=0$ also for $\epsilon=0$. In the following Proposition~\ref{L0} we study equation \eqref{L=0}  for $\epsilon=0$ and $\theta \in \{-1,1\}$.

\begin{prop}\label{L0}
Let $\Omega^i$, $\Omega^o$ be as in \eqref{e1}.  Let $\theta\in\{-1,1\}$. Let $(f^i, f^o)\in C^{1,\alpha}(\partial\Omega^i)\times C^{1,\alpha}(\partial\Omega^o)$.  Then, there exists a unique pair of functions $(\mu^i, \mu^o)\in C^{0,\alpha}(\partial\Omega^i)\times C^{0,\alpha}(\partial\Omega^o)$ such that 
\begin{equation}\label{L0=0}
\Lambda_\theta[0,f^i,f^o,\mu^i,\mu^o]=(0,0)\,.
\end{equation} Moreover, the function $u\equiv v[\mu^o]_{|\mathrm{cl}\Omega^o}$  is the unique solution in $C^{1,\alpha}(\mathrm{cl}\Omega^o)$ of the boundary value problem 
\[
\left\{
\begin{array}{ll}
\Delta u=0&\text{ in }\Omega^o\,,\\
u=f^o&\text{ on }\partial\Omega^o\,.
\end{array}
\right.
\] \end{prop}
\begin{proof} We observe that the equation in \eqref{L0=0} is equivalent to the following system of equations
\[
\left\{
\begin{array}{ll}
\theta v[\mu^i]_{|\partial\Omega^i}+v[\mu^o](0)=f^i&\text{ on }\partial\Omega^i\,,\\
v[\mu^o]_{|\partial\Omega^o}=f^o&\text{ on }\partial\Omega^o\,.
\end{array}
\right.
\] 
Then the validity of the Lemma can be deduced by Lemma~\ref{homeo}.
\end{proof}

\medskip

 In the following Propositions~\ref{Mr}, \ref{Meven} and \ref{Modd} we exploit the Implicit Function Theorem  for real analytic maps to investigate the dependence of the solution $(\mu^i,\mu^o)$ of the equations in \eqref{L=0} and \eqref{L0=0} upon $(\epsilon,f^i,f^o)$. In particular, in Proposition~\ref{Mr} we study what happens for $\epsilon$ small, while in Propositions~\ref{Meven} and \ref{Modd} we consider the case of dimension $n$ even and odd, respectively.  To prove Propositions~\ref{Mr}, \ref{Meven} and \ref{Modd} we need
 to analyze the regularity of the operator $\Lambda_\theta$. The definition of $\Lambda_\theta$ involves the single layer potential $v[\mu]$ and also integral operators which display no singularity. To analyze their regularity we  need the following Lemma~\ref{anal}.

\begin{lem}\label{anal}
Let $\Omega$, $\tilde\Omega$ be open bounded subsets of $\mathbb{R}^n$ of class $C^{1,\alpha}$. Then the following statements hold. 
\begin{enumerate}
\item[(i)] The map $G$ from $\{(\psi,\phi,\mu)\in C^{1,\alpha}(\partial\tilde\Omega,\mathbb{R}^n)\times C^{1,\alpha}(\partial\Omega,\mathbb{R}^n)\times C^{0,\alpha}(\partial\Omega)\,:\,\psi(\partial\tilde\Omega)\cap\phi(\partial\Omega)=\emptyset\}
$ to $C^{1,\alpha}(\partial\tilde\Omega)$ which takes $(\psi,\phi,\mu)$ to the function $G[\psi,\phi,\mu]$ defined by
\[
G[\psi,\phi,\mu](x)\equiv\int_{\partial\Omega}S_n(\psi(x)-\phi(y))\mu(y)\,d\sigma_y\qquad\forall x\in\partial\tilde\Omega\,,
\] is real analytic.
\item[(ii)] The map $H$ from $\{(\Phi,\phi,\mu)\in C^{1,\alpha}(\mathrm{cl}\tilde\Omega,\mathbb{R}^n)\times C^{1,\alpha}(\partial\Omega,\mathbb{R}^n)\times C^{0,\alpha}(\partial\Omega)\,:\,\Phi(\mathrm{cl}\tilde\Omega)\cap\phi(\partial\Omega)=\emptyset\}
$ to $C^{1,\alpha}(\mathrm{cl}\tilde\Omega)$ which takes $(\Phi,\phi,\mu)$ to the function $H[\Phi,\phi,\mu]$ defined by
\[
H[\Phi,\phi,\mu](x)\equiv\int_{\partial\Omega}S_n(\Phi(x)-\phi(y))\mu(y)\,d\sigma_y\qquad\forall x\in\mathrm{cl}\tilde\Omega\,,
\] is real analytic.
\end{enumerate}
\end{lem} 
\begin{proof} The validity of the Lemma follows immediately by the results of Lanza de Cristoforis and the second author in~\cite{LaMu10b}. See also Lanza de Cristoforis~\cite[Theorem~6.2]{Lan07a} where a known result on composition operators has been exploited (cf.~B{\"o}hme and Tomi~\cite[p.10]{BT73}, Henry~\cite[p.29]{Hen82}, Valent~\cite[Theorem~5.2, p.44]{Val88}.)\end{proof}

\begin{prop}\label{Mr}
Let $\Omega^i$, $\Omega^o$ be as in \eqref{e1}. Let $\epsilon_0$ be as in \eqref{e2}. Let $\theta\in\{-1,1\}$. Let $(\tilde f^i,\tilde f^o)\in C^{1,\alpha}(\partial\Omega^i)\times C^{1,\alpha}(\partial\Omega^o)$. Let the pair $(\tilde\mu^i,\tilde\mu^o)$ be the unique solution  in $C^{0,\alpha}(\partial\Omega^i)\times C^{0,\alpha}(\partial\Omega^o)$ of $\Lambda_\theta[0,\tilde f^i,\tilde f^o,\tilde\mu^i,\tilde\mu^o]=0$. Then there exist $\tilde\epsilon$ in $]0,\epsilon_0[$, and an open neighborhood $\mathcal{U}$ of $(\tilde f^i,\tilde f^o)$ in $C^{1,\alpha}(\partial\Omega^i)\times C^{1,\alpha}(\partial\Omega^o)$, and an open neighborhood $\mathcal{V}$ of $(\tilde\mu^i,\tilde\mu^o)$ in $C^{0,\alpha}(\partial\Omega^i)\times C^{0,\alpha}(\partial\Omega^o)$, and a real analytic operator $\tilde{M}_\theta\equiv(\tilde{M}_\theta^i,\tilde{M}_\theta^o)$ from $]-\tilde\epsilon,\tilde\epsilon[\times\mathcal{U}$ to $\mathcal{V}$ such that the set of zeros of $\Lambda_\theta$ in $]-\tilde\epsilon,\tilde\epsilon[\times\mathcal{U}\times\mathcal{V}$ coincides with the graph of $\tilde{M}_\theta$. In particular,
\begin{equation}\label{Mr_eqn1}
\Lambda_\theta[\epsilon, f^i, f^o,\tilde{M}_\theta[\epsilon,f^i,f^o]]=(0,0)\qquad\forall (\epsilon,f^i,f^o)\in ]-\tilde\epsilon,\tilde\epsilon[\times\mathcal{U}\,.
\end{equation}
\end{prop}
\begin{proof}  We note that the existence and uniqueness of the solution $(\tilde\mu^i,\tilde\mu^o)$ follows by Proposition~\ref{L0}. We  now prove the statement by applying the Implicit Function Theorem for real analytic maps to the equation in \eqref{Mr_eqn1} around $(0,\tilde f^i,\tilde f^o,\tilde\mu^i,\tilde\mu^o)$.  To do so, we first show  that $\Lambda_\theta$ is real analytic from $]-\epsilon_0,\epsilon_0[\times C^{1,\alpha}(\partial\Omega^i)\times C^{1,\alpha}(\partial\Omega^o)\times C^{0,\alpha}(\partial\Omega^i)\times C^{0,\alpha}(\partial\Omega^o)$ to $C^{1,\alpha}(\partial\Omega^i)\times C^{1,\alpha}(\partial\Omega^o)$. By Lemma~\ref{anal} (i), the map from $]-\epsilon_0,\epsilon_0[\times C^{0,\alpha}(\partial\Omega^o)$ to $C^{1,\alpha}(\partial\Omega^i)$ which takes $(\epsilon,\mu^o)$ to the function $\int_{\partial\Omega^o}S_n(\epsilon x-y)\,\mu^o(y)\,d\sigma_y$ of $x\in\partial\Omega^i$ is real analytic. Lemma~\ref{homeo} implies that the map from $C^{0,\alpha}(\partial\Omega^i)$ to  $C^{1,\alpha}(\partial\Omega^i)$ which takes $\mu^i$ to the function $\int_{\partial\Omega^i}S_n(x-y)\,\mu^i(y)\,d\sigma_y$ of $x\in\partial\Omega^i$ is real analytic.  Then, by  standard calculus in Banach space we deduce that $\Lambda^i_\theta$ is real analytic from $]-\epsilon_0,\epsilon_0[\times C^{1,\alpha}(\partial\Omega^i)\times C^{1,\alpha}(\partial\Omega^o)\times C^{0,\alpha}(\partial\Omega^i)\times C^{0,\alpha}(\partial\Omega^o)$ to $C^{1,\alpha}(\partial\Omega^i)$. By a similar argument we can show that $\Lambda^o_\theta$ is real analytic from $]-\epsilon_0,\epsilon_0[\times C^{1,\alpha}(\partial\Omega^i)\times C^{1,\alpha}(\partial\Omega^o)\times C^{0,\alpha}(\partial\Omega^i)\times C^{0,\alpha}(\partial\Omega^o)$ to $C^{1,\alpha}(\partial\Omega^o)$. Hence $\Lambda_\theta$ is real analytic. Now we observe that the partial differential  of $\Lambda_\theta$ at $(0,\tilde f^i,\tilde f^o,\tilde\mu^i,\tilde\mu^o)$ with respect to the variables $(\mu^i,\mu^o)$ is delivered by the following formulas
\begin{eqnarray}\label{0721eqn1}
\lefteqn{\partial_{(\mu^i,\mu^o)}\Lambda_\theta^i[0,\tilde f^i ,\tilde f^o,\tilde\mu^i,\tilde\mu^o](\bar{\mu}^i,\bar{\mu}^o)(x)}\\
\nonumber
&&=\theta\int_{\partial \Omega^i}S_n(x- y)\bar{\mu}^i(y)\,d\sigma_y
+\int_{\partial \Omega^o}S_n(y)\bar{\mu}^o(y)\,d\sigma_y\ \qquad\qquad \forall x \in \partial \Omega^i\,,\\
\nonumber
\lefteqn{\partial_{(\mu^i,\mu^o)}\Lambda_\theta^o[0,\tilde f^i,\tilde f^o,\tilde\mu^i,\tilde\mu^o](\bar{\mu}^i,\bar{\mu}^o)(x)=\int_{\partial \Omega^o}S_n(x-y)\bar{\mu}^o(y)\,d\sigma_y\quad \forall x \in \partial\Omega^o}
\end{eqnarray}
for all $(\bar{\mu}^i,\bar{\mu}^o)\in C^{0,\alpha}(\partial \Omega^i) \times C^{0,\alpha}(\partial \Omega^o)$. We have to show that the differential $\partial_{(\mu^i,\mu^o)}\Lambda_\theta[0,\tilde f^i,\tilde f^o,\tilde\mu^i,\tilde\mu^o]$ is a linear homeomorphism. By the Open Mapping Theorem, it suffices to show that it is a bijection from $C^{0,\alpha}(\partial \Omega^i)\times C^{0,\alpha}(\partial \Omega^o)$ to $C^{1,\alpha}(\partial \Omega^i)\times C^{1,\alpha}(\partial \Omega^o)$. Let $(\bar{f}^i,\bar{f}^o)\in C^{1,\alpha}(\partial \Omega^i)\times C^{1,\alpha}(\partial \Omega^o)$. By the equalities in \eqref{0721eqn1} and by  Lemma~\ref{homeo} we deduce that there exists a unique pair $(\bar{\mu}^i,\bar{\mu}^o)\in C^{0,\alpha}(\partial \Omega^i)\times C^{0,\alpha}(\partial \Omega^o)$ such that
\[
\partial_{(\mu^i,\mu^o)}\Lambda_\theta[0,\tilde f^i,\tilde f^o,\tilde\mu^i,\tilde\mu^o](\bar{\mu}^i,\bar{\mu}^o)=(\bar{f}^i,\bar{f}^o)\,
\] (see also the proof of Lemma~\ref{L0}.)
Hence we can invoke the Implicit Function Theorem for real analytic maps in Banach spaces and deduce the existence of $\tilde\epsilon$, $\mathcal{U}$, $\mathcal{V}$, $\tilde M_{\theta}$ as in the statement.
\end{proof}

\begin{prop}\label{Meven}
Let $\Omega^i$, $\Omega^o$ be as in \eqref{e1}. Let $\epsilon_0$ be as in \eqref{e2}. If the dimension $n$ is even, then there exists a real analytic map $M\equiv(M^i,M^o)$ from $]-\epsilon_0,\epsilon_0[\times C^{1,\alpha}(\partial\Omega^i)\times C^{1,\alpha}(\partial\Omega^o)$ to $C^{0,\alpha}(\partial\Omega^i)\times C^{0,\alpha}(\partial\Omega^o)$ such that 
\begin{equation}\label{LM=0}
\Lambda_1[\epsilon,f^i,f^o,M[\epsilon,f^i,f^o]]=(0,0)
\end{equation} for all $(\epsilon,f^i,f^o)\in ]-\epsilon_0,\epsilon_0[\times C^{1,\alpha}(\partial\Omega^i)\times C^{1,\alpha}(\partial\Omega^o)$.
\end{prop}
\begin{proof} By Propositions~\ref{L} and \ref{L0} we deduce that there exists a unique map $M$ from $]-\epsilon_0,\epsilon_0[\times C^{1,\alpha}(\partial\Omega^i)\times C^{1,\alpha}(\partial\Omega^o)$ to $C^{0,\alpha}(\partial\Omega^i)\times C^{0,\alpha}(\partial\Omega^o)$ which satisfies \eqref{LM=0}. We show that $M$ is real analytic by exploiting the Implicit Function Theorem for real analytic maps.  By Lemmas~\ref{homeo} and \ref{anal} and by standard calculus in Banach space we verify that $\Lambda_1$ is real analytic from $]-\epsilon_0,\epsilon_0[\times C^{1,\alpha}(\partial\Omega^i)\times C^{1,\alpha}(\partial\Omega^o)\times C^{0,\alpha}(\partial\Omega^i)\times C^{0,\alpha}(\partial\Omega^o)$ to $C^{1,\alpha}(\partial\Omega^i)\times C^{1,\alpha}(\partial\Omega^o)$ (see also the proof of Proposition~\ref{Mr}.) By the Implicit Function Theorem for real analytic maps, it clearly suffices to prove that if $(\epsilon,f^i,f^o)$ is in $]-\epsilon_0,\epsilon_0[\times C^{1,\alpha}(\partial\Omega^i)\times C^{1,\alpha}(\partial\Omega^o)$, then the partial differential of $\Lambda_1$ at $(\epsilon,f^i, f^o,M[\epsilon,f^i,f^o])$ with respect to the variables $(\mu^i,\mu^o)$ is a linear homeomorphism from $C^{0,\alpha}(\partial \Omega^i)\times C^{0,\alpha}(\partial \Omega^o)$ onto $C^{1,\alpha}(\partial \Omega^i)\times C^{1,\alpha}(\partial \Omega^o)$. By Proposition \ref{Mr}, we can confine  ourselves to consider $(\epsilon,f^i,f^o)$ in $\bigl(]-\epsilon_0,\epsilon_0[\setminus\{0\}\bigr)\times C^{1,\alpha}(\partial\Omega^i)\times C^{1,\alpha}(\partial\Omega^o)$. By standard calculus in Banach space, the partial differential $\partial_{(\mu^i,\mu^o)}\Lambda_1[\epsilon,f^i, f^o,M[\epsilon,f^i,f^o]]$  is delivered by the following formulas
\begin{eqnarray*}
\lefteqn{\partial_{(\mu^i,\mu^o)}\Lambda_1^i[\epsilon, f^i, f^o,M[\epsilon,f^i,f^o]](\bar{\mu}^i,\bar{\mu}^o)(x)}\\
&&= \int_{\partial \Omega^i}S_n(x- y)\bar{\mu}^i(y)\,d\sigma_y
+\int_{\partial \Omega^o}S_n(\epsilon x-y)\bar{\mu}^o(y)\,d\sigma_y\quad\qquad \forall x \in \partial \Omega^i\, ,\\
\lefteqn{\partial_{(\mu^i,\mu^o)}\Lambda_1^o[\epsilon, f^i, f^o,M[\epsilon,f^i,f^o]](\bar{\mu}^i,\bar{\mu}^o)(x)}\\ 
&&=\epsilon^{n-2}\int_{\partial \Omega^i}S_n(x-\epsilon y)\bar{\mu}^i(y)\,d\sigma_y+\int_{\partial \Omega^o}S_n(x-y)\bar{\mu}^o(y)\,d\sigma_y\quad \forall x \in \partial\Omega^o\nonumber
\end{eqnarray*}
for all $(\bar{\mu}^i,\bar{\mu}^o)\in C^{0,\alpha}(\partial \Omega^i) \times C^{0,\alpha}(\partial \Omega^o)$. Then by Lemma~\ref{homeo} and by the Open Mapping Theorem, we deduce that $\partial_{(\mu^i,\mu^o)}\Lambda_1[\epsilon, f^i, f^o,M[\epsilon,f^i,f^o]]$ is a linear homeomorphism from $C^{0,\alpha}(\partial \Omega^i)\times C^{0,\alpha}(\partial \Omega^o)$ onto $C^{1,\alpha}(\partial \Omega^i)\times C^{1,\alpha}(\partial \Omega^o)$. The proof of the Proposition is now complete.\end{proof}

\begin{prop}\label{Modd}
Let $\Omega^i$, $\Omega^o$ be as in \eqref{e1}. Let $\epsilon_0$ be as in \eqref{e2}. If the dimension $n$ is odd, then there exist  real analytic maps $M_+\equiv(M^i_+,M^o_+)$ from $]0,\epsilon_0[\times C^{1,\alpha}(\partial\Omega^i)\times C^{1,\alpha}(\partial\Omega^o)$ to $C^{0,\alpha}(\partial\Omega^i)\times C^{0,\alpha}(\partial\Omega^o)$ and $M_-\equiv(M^i_-,M^o_-)$ from $]-\epsilon_0,0[\times C^{1,\alpha}(\partial\Omega^i)\times C^{1,\alpha}(\partial\Omega^o)$ to  $C^{0,\alpha}(\partial\Omega^i)\times C^{0,\alpha}(\partial\Omega^o)$ such that 
\begin{equation*}
\Lambda_1[\epsilon,f^i,f^o,M_+[\epsilon,f^i,f^o]]=(0,0)
\end{equation*} for all $(\epsilon,f^i,f^o)\in ]0,\epsilon_0[\times C^{1,\alpha}(\partial\Omega^i)\times C^{1,\alpha}(\partial\Omega^o)$, and such that 
\begin{equation*}
\Lambda_{-1}[\epsilon,f^i,f^o,M_-[\epsilon,f^i,f^o]]=(0,0)
\end{equation*} for all $(\epsilon,f^i,f^o)\in ]-\epsilon_0,0[\times C^{1,\alpha}(\partial\Omega^i)\times C^{1,\alpha}(\partial\Omega^o)$. 
\end{prop}
\begin{proof} It is a slight modification of the proof of Proposition~\ref{Meven} and is accordingly omitted.\end{proof}

\section{Main results for real analytic families of harmonic functions on $\Omega(\epsilon)$}\label{main}

We prove in this section our main Theorems~\ref{even} and \ref{odd}. In Theorem~\ref{even} we consider the case of dimension $n$ even. We note that Theorem~\ref{even} implies the validity of statement (j) in Section~\ref{intro}.

\begin{thm}\label{even}
Assume that the dimension $n$ is even. Let $\Omega^i$, $\Omega^o$ be as in \eqref{e1}. Let $\epsilon_0$ be as in \eqref{e2}. Let $\epsilon_1\in]0,\epsilon_0]$. Let $\{u_\epsilon\}_{\epsilon\in]0,\epsilon_1[}$ be a family of functions which satisfies the condition in (a0) and such that
\begin{enumerate}
\item[(i)] there exists a real analytic operator $B^o$ from $]-\epsilon_1,\epsilon_1[$ to $C^{1,\alpha}(\partial\Omega^o)$ such that $u_{\epsilon}(x)=B^o[\epsilon](x)$ for all $x\in\partial\Omega^o$ and all $\epsilon\in]0,\epsilon_1[$,
\item[(ii)] there exists a real analytic operator $B^i$ from $]-\epsilon_1,\epsilon_1[$ to $C^{1,\alpha}(\partial\Omega^i)$ such that $u_{\epsilon}(\epsilon x)=B^i[\epsilon](x)$ for all $x\in\partial\Omega^i$ and all $\epsilon\in]0,\epsilon_1[$. 
\end{enumerate} Then there exists a family of functions $\{v_\epsilon\}_{\epsilon\in]-\epsilon_1,\epsilon_1[}$ which satisfies the conditions in (b0)--(b2) and such that $u_\epsilon=v_\epsilon$ for all $\epsilon\in]0,\epsilon_1[$.
\end{thm}
\begin{proof}  Let $M\equiv(M^i,M^o)$ be the map in Proposition~\ref{Meven}. We set
\[
v^i_\epsilon(x)\equiv \epsilon^{n-2}\int_{\partial\Omega^i}S_n(x - \epsilon y)M^i[\epsilon,B^i[\epsilon],B^o[\epsilon]](y)\,d\sigma_y\,\qquad\forall x\in\mathrm{cl}\Omega(\epsilon) \, ,
\]
for all $\epsilon \in ]-\epsilon_1,\epsilon_1[\setminus \{0\}$, and $v^i_0(x)\equiv 0$ for all $x \in \mathrm{cl}\Omega^o$.
Then we set
\begin{eqnarray*}
v^o_\epsilon(x)&\equiv& \int_{\partial\Omega^o}S_n(x - y)M^o[\epsilon,B^i[\epsilon],B^o[\epsilon]](y)\,d\sigma_y\,, \\
v_\epsilon(x)&\equiv& v^i_\epsilon(x) + v^o_\epsilon(x)\,,\qquad\qquad\qquad\qquad\forall x\in\mathrm{cl}\Omega(\epsilon)\, , 
\end{eqnarray*} for all $\epsilon\in]-\epsilon_1,\epsilon_1[$.
By classical Potential Theory, we deduce that $\{v_\epsilon\}_{\epsilon\in]-\epsilon_1,\epsilon_1[}$ satisfies the condition in (b0). Now let $\Omega_M$, $\epsilon_M$ be as in (b1). Let $V_M\equiv v_{\epsilon|\mathrm{cl}\Omega_M}$ for all $\epsilon\in]-\epsilon_M,\epsilon_M[$. We show that $V_M$ is real analytic and hence  $\{v_\epsilon\}_{\epsilon\in]-\epsilon_1,\epsilon_1[}$ satisfies the condition in (b1).  To do so we prove that $V_M$ is real analytic in a neighborhood of a fixed point $\epsilon^*$ of $]-\epsilon_M,\epsilon_M[$.  We note that $\mathrm{cl}\Omega_M\cap\epsilon^*\mathrm{cl}\Omega^i=\emptyset$. Then, by a standard argument based on the existence of smooth Partitions of Unity and on Sard's Theorem we can show that there exists an open bounded set $\tilde\Omega$ of class $C^{1,\alpha}$ such that $\Omega_M\subseteq\tilde\Omega\subseteq\Omega^o$ and $\mathrm{cl}\tilde\Omega\cap\epsilon^*\mathrm{cl}\Omega^i=\emptyset$.
Then, by the continuity of the real function which takes $\epsilon$ to $\mathrm{dist}(\epsilon{\mathrm{cl}\Omega^i}\,,\,\mathrm{cl}\tilde\Omega)\equiv\mathrm{inf}\{|x-y|\,:\,x\in\epsilon{\mathrm{cl}\Omega^i}\,,\,y\in\mathrm{cl}\tilde\Omega\}$ we deduce that there exists $\delta>0$ such that  $\epsilon\mathrm{cl}\Omega^i\cap\mathrm{cl}\tilde\Omega=\emptyset$ for all $\epsilon\in ]\epsilon^*-\delta,\epsilon^*+\delta[$. Possibly shrinking $\delta$ we can assume that $]\epsilon^*-\delta,\epsilon^*+\delta[\subseteq]-\epsilon_1,\epsilon_1[$. Then, by Lemma~\ref{anal}~(ii) and by the real analyticity of $M$, $B^i$, $B^o$ we verify that the map from $]\epsilon^*-\delta,\epsilon^*+\delta[$ to $C^{1,\alpha}(\mathrm{cl}\tilde\Omega)$ which takes $\epsilon$ to $v^i_{\epsilon|\mathrm{cl}\tilde\Omega}$ is real analytic. Then, by the boundedness of the restriction operator from $C^{1,\alpha}(\mathrm{cl}\tilde\Omega)$ to $C^{1,\alpha}(\mathrm{cl}\Omega_M)$ and by standard calculus in Banach space, the map from $]\epsilon^*-\delta,\epsilon^*+\delta[$ to $C^{1,\alpha}(\mathrm{cl}\Omega_M)$ which takes $\epsilon$ to $v^i_{\epsilon|\mathrm{cl}\Omega_M}$ is real analytic.  By Lemma~\ref{homeo}, and by the real analyticity of $M$, $B^i$, $B^o$, and by the boundedness of the restriction operator from $C^{1,\alpha}(\mathrm{cl}\Omega^o)$ to $C^{1,\alpha}(\mathrm{cl}\Omega_M)$ we deduce that the map from $]\epsilon^*-\delta,\epsilon^*+\delta[$ to $C^{1,\alpha}(\mathrm{cl}\Omega_M)$ which takes $\epsilon$ to $v^o_{\epsilon|\mathrm{cl}\Omega_M}$ is real analytic. Then  the map from $]\epsilon^*-\delta,\epsilon^*+\delta[$ to $C^{1,\alpha}(\mathrm{cl}\Omega_M)$ which takes $\epsilon$ to $V_M[\epsilon]=v^i_{\epsilon|\mathrm{cl}\Omega_M}+v^o_{\epsilon|\mathrm{cl}\Omega_M}$ is real analytic. Thus, $\{v_\epsilon\}_{\epsilon\in]-\epsilon_1,\epsilon_1[}$ satisfies the conditions in (b1).  Now we prove that  $\{v_\epsilon\}_{\epsilon\in]-\epsilon_1,\epsilon_1[}$ satisfies the conditions in (b2). Let $\Omega_m$ and $\epsilon_m$ be as in (b2). Let $V_m[\epsilon]$ be defined by
\begin{eqnarray*}
\lefteqn{V_m[\epsilon](x)\equiv \int_{\partial\Omega^i}S_n(x -  y) M^i[\epsilon,B^i[\epsilon],B^o[\epsilon]](y)\,d\sigma_y}\\
&&
\qquad\quad 
+\int_{\partial\Omega^o}S_n(\epsilon x - y) M^o[\epsilon,B^i[\epsilon],B^o[\epsilon]](y)\,d\sigma_y\qquad\forall x\in\mathrm{cl}\Omega_m
\end{eqnarray*} for all $\epsilon\in]-\epsilon_m,\epsilon_m[$. Clearly,
\begin{equation}\label{odd_eqn1}
v_\epsilon (\epsilon x)=V_m[\epsilon](x) \qquad \forall x \in  \mathrm{cl}\Omega_m\,,\  \epsilon \in ]-\epsilon_m,\epsilon_m[\setminus \{0\}\,.
\end{equation}
We prove that the map from $]-\epsilon_m,\epsilon_m[$ to $C^{1,\alpha}(\mathrm{cl}\Omega_m)$ which takes $\epsilon$ to $V_m[\epsilon]$ is real analytic.  To do so we prove that $V_m$ is real analytic in a neighborhood of a fixed point $\epsilon^*$ of $]-\epsilon_m,\epsilon_m[$. By a standard argument based on the existence of smooth Partitions of Unity, and on Sard's Theorem, and on the continuity of the distance function,  we verify that there exist $\delta>0$ and an open bounded subset $\tilde\Omega$ of $\mathbb{R}^n\setminus\mathrm{cl}\Omega^i$ of class $C^{1,\alpha}$ such that $\Omega_m\subseteq\tilde\Omega$ and $\epsilon\mathrm{cl}\tilde\Omega\subseteq\Omega^o$ for all $\epsilon\in ]\epsilon^*-\delta,\epsilon^*+\delta[$. Possibly shrinking $\delta$ we can assume that $]\epsilon^*-\delta,\epsilon^*+\delta[\subseteq]-\epsilon_m,\epsilon_m[$. Then we set
\begin{eqnarray*}
\lefteqn{\tilde V_m[\epsilon](x)\equiv \int_{\partial\Omega^i}S_n(x -  y) M^i[\epsilon,B^i[\epsilon],B^o[\epsilon]](y)\,d\sigma_y}\\
&&
\qquad\quad
+\int_{\partial\Omega^o}S_n(\epsilon x - y) M^o[\epsilon,B^i[\epsilon],B^o[\epsilon]](y)\,d\sigma_y\qquad\forall x\in\mathrm{cl}\tilde\Omega
\end{eqnarray*} for all $\epsilon\in]\epsilon^*-\delta,\epsilon^*+\delta[$.  So that $V_m[\epsilon]=\tilde V_m[\epsilon]_{|\mathrm{cl}\Omega_m}$ for all $\epsilon\in]\epsilon^*-\delta,\epsilon^*+\delta[$. Then, by Lemma~\ref{homeo}, and by Lemma~\ref{anal}~(ii), and by the real analyticity of $M$,  and by standard calculus in Banach space, we deduce that $\tilde V_m$ is real analytic from $]\epsilon^*-\delta,\epsilon^*+\delta[$ to $C^{1,\alpha}(\mathrm{cl}\tilde\Omega)$. Then, by the boundedness of the restriction operator from $C^{1,\alpha}(\mathrm{cl}\tilde\Omega)$ to $C^{1,\alpha}(\mathrm{cl}\Omega_m)$, $V_m$ is real analytic from $]\epsilon^*-\delta,\epsilon^*+\delta[$ to $C^{1,\alpha}(\mathrm{cl}\Omega_m)$. Thus, the validity of (b2) follows. Moreover, by Proposition~\ref{L} and by the uniqueness of the solution of the Dirichlet boundary value problem in $\Omega(\epsilon)$ we deduce that  $u_\epsilon=v_\epsilon$ for $\epsilon\in]0,\epsilon_1[$. The validity of the Theorem is now verified.\end{proof}

\medskip

We now consider the case of dimension $n$ odd and we prove our main Theorem~\ref{odd}. We note that Theorem~\ref{odd} implies the validity of statement (jj) in Section~\ref{intro}.

\begin{thm}\label{odd}
Assume that the dimension $n$ is odd. Let $\Omega^i$, $\Omega^o$ be as in \eqref{e1}. Let $\epsilon_0$ be as in \eqref{e2}. Let $\epsilon_1\in]0,\epsilon_0]$. Let $\{v_\epsilon\}_{\epsilon\in]-\epsilon_1,\epsilon_1[}$ be a family of functions which satisfies the condition in (b0) and such that
\begin{enumerate}
\item[(i)] there exists a real analytic operator $B^o$ from $]-\epsilon_1,\epsilon_1[$ to $C^{1,\alpha}(\partial\Omega^o)$ such that $v_{\epsilon}(x)=B^o[\epsilon](x)$ for all $x\in\partial\Omega^o$ and all $\epsilon\in]-\epsilon_1,\epsilon_1[$,
\item[(ii)] there exists a real analytic operator $B^i$ from $]-\epsilon_1,\epsilon_1[$ to $C^{1,\alpha}(\partial\Omega^i)$ such that $v_{\epsilon}(\epsilon x)=B^i[\epsilon](x)$ for all $x\in\partial\Omega^i$ and all $\epsilon\in]-\epsilon_1,\epsilon_1[\setminus\{0\}$.
\end{enumerate}
Assume that the family $\{v_\epsilon\}_{\epsilon\in]-\epsilon_1,\epsilon_1[}$ satisfies at least one of the following conditions (iii) and (iv).
\begin{enumerate}
\item[(iii)] There exist an open non-empty subset $\Omega_M$ of $\Omega^o\setminus\{0\}$, and a real number $\epsilon_M\in]0,\epsilon_1]$ such that  $\mathrm{cl}\Omega_M\cap\epsilon\mathrm{cl}\Omega^i=\emptyset$ for all $\epsilon\in]-\epsilon_M,\epsilon_M[$, and a real analytic operator  $V_M$ from $]-\epsilon_M,\epsilon_M[$ to $C^{1,\alpha}(\mathrm{cl}\Omega_M)$ such that \[
v_{\epsilon|\mathrm{cl}\Omega_M}=V_M[\epsilon]\qquad\forall\epsilon\in]-\epsilon_M,\epsilon_M[\,. 
\]
\item[(iv)] There exist a bounded open non-empty subset $\Omega_m$ of $\mathbb{R}^n\setminus\mathrm{cl}\Omega^i$, and a real number $\epsilon_m\in]0,\epsilon_1]$ such that  $\epsilon\mathrm{cl}\Omega_m\subseteq\Omega^o$ for all $\epsilon\in]-\epsilon_m,\epsilon_m[$, and a real analytic operator $V_m$ from $]-\epsilon_m,\epsilon_m[$ to $C^{1,\alpha}(\mathrm{cl}\Omega_m)$  such that
\[
v_{\epsilon}(\epsilon\,\cdot\,)_{|\mathrm{cl}\Omega_m}=V_m[\epsilon]\qquad\forall\epsilon\in]-\epsilon_m,\epsilon_m[\setminus \{0\}\,. 
\]
\end{enumerate} Then there exists a family of functions $\{w_\epsilon\}_{\epsilon\in]-\epsilon_1,\epsilon_1[}$ which satisfies the conditions in (c0), (c1) and such that $v_\epsilon=w_{\epsilon|\mathrm{cl}\Omega(\epsilon)}$ for all $\epsilon\in]-\epsilon_1,\epsilon_1[$.
\end{thm}
\begin{proof}  Let $\tilde M_1\equiv(\tilde M^i_1,\tilde M^o_1)$, $\tilde\epsilon$, $\mathcal{U}$ be as in Propositions~\ref{Mr} with $\theta\equiv 1$, $\tilde f^i\equiv B^i[0]$ and $\tilde f^o\equiv B^o[0]$. We show that $\tilde M^i_1[\epsilon,B^i[\epsilon], B^o[\epsilon]]=0$ for $\epsilon$ in an open neighborhood of $0$. To do so we first prove that both conditions (iii) and (iv) imply that there exists  $\tilde\epsilon_*\in]0,\tilde\epsilon]$ such that 
\begin{equation}\label{110802_eqn1}
\int_{\partial\Omega^i}S_n(x -  y)\tilde M^i_1[\epsilon,B^i[\epsilon],B^o[\epsilon]](y)\,d\sigma_y=0\quad\forall x\in\partial\Omega^i\,,\ \epsilon\in]-\tilde\epsilon_*,0[\,.
\end{equation}

Assume that $\{v_\epsilon\}_{\epsilon\in]-\epsilon_1,\epsilon_1[}$ satisfies the condition in (iii). We can take $\tilde\epsilon_M\in]0,\inf\{\epsilon_M, \tilde\epsilon\}]$ such that $(B^i[\epsilon],B^o[\epsilon])\in\mathcal{U}$ for all $\epsilon\in]-\tilde\epsilon_M,\tilde\epsilon_M[$. Then we set
\begin{eqnarray}\label{110304_eqn3}
\tilde v_\epsilon(x)&\equiv& \epsilon^{n-2}\int_{\partial\Omega^i}S_n(x - \epsilon y)\tilde M^i_1[\epsilon,B^i[\epsilon],B^o[\epsilon]](y)\,d\sigma_y\\
\nonumber
&&
+\int_{\partial\Omega^o}S_n(x - y)\tilde M^o_1[\epsilon,B^i[\epsilon],B^o[\epsilon]](y)\,d\sigma_y\qquad\forall x\in\mathrm{cl}\Omega(\epsilon)\,,
\end{eqnarray} 
for all $\epsilon \in ]-\tilde\epsilon_M,\tilde\epsilon_M[ \setminus \{0\}$,  and
\begin{equation}\label{110606_eqn1}
\tilde{v}_0(x)\equiv \int_{\partial\Omega^o}S_n(x - y)\tilde M^o_1[0,B^i[0],B^o[0]](y)\,d\sigma_y\qquad\forall x\in\mathrm{cl}\Omega^o\, .
\end{equation}
Then Propositions~\ref{L} and \ref{L0} imply that $\tilde v_\epsilon=v_\epsilon$ for all $\epsilon\in]0,\tilde\epsilon_M[$. So that $\tilde v_{\epsilon|\mathrm{cl}\Omega_M}=v_{\epsilon|\mathrm{cl}\Omega_M}$ for all $\epsilon\in]0,\tilde\epsilon_M[$. We observe that the map from $]-\tilde\epsilon_M,\tilde\epsilon_M[$ to $C^{1,\alpha}(\mathrm{cl}\Omega_M)$ which takes $\epsilon$ to $\tilde v_{\epsilon|\mathrm{cl}\Omega_M}$ is real analytic (see also the argument developed in the proof of Theorem~\ref{even} for $V_M$.) Then, by the assumption in (iii) and by the Identity Principle for real analytic maps, we have $\tilde v_{\epsilon|\mathrm{cl}\Omega_M}=v_{\epsilon|\mathrm{cl}\Omega_M}$ for all $\epsilon\in]-\tilde\epsilon_M,\tilde\epsilon_M[$. We note that $\tilde v_\epsilon$ is harmonic on $\Omega(\epsilon)$ for all $\epsilon\in]-\tilde\epsilon_M,\tilde\epsilon_M[$. Thus, the equality $\tilde v_{\epsilon|\mathrm{cl}\Omega_M}=v_{\epsilon|\mathrm{cl}\Omega_M}$ implies that $\tilde v_{\epsilon}=v_{\epsilon}$ on the whole of $\mathrm{cl}\Omega(\epsilon)$ for all $\epsilon\in]-\tilde\epsilon_M,\tilde\epsilon_M[$. In particular, 
\[
\tilde v_{\epsilon}(\epsilon x)=v_{\epsilon}(\epsilon x)\qquad\forall x\in\partial\Omega^i\,,\ \epsilon\in]-\tilde\epsilon_M,0[\,,
\] which in turn implies that
\begin{equation}\label{110304_eqn1}
\begin{split}
&{-\int_{\partial\Omega^i}S_n(x -  y)\tilde M^i_1[\epsilon,B^i[\epsilon],B^o[\epsilon]](y)\,d\sigma_y}\\
&\quad
+\int_{\partial\Omega^o}S_n(\epsilon x - y)\tilde M^o_1[\epsilon,B^i[\epsilon],B^o[\epsilon]](y)\,d\sigma_y=B^i[\epsilon](x)
\end{split}
\end{equation}  for all  $x\in\partial\Omega^i$, $\epsilon\in]-\tilde\epsilon_M,0[$.
By the definition of $\tilde M_1$ in Proposition~\ref{Mr} we have $\Lambda_1[\epsilon,B^i[\epsilon],B^o[\epsilon],\tilde M_1[\epsilon,B^i[\epsilon],B^o[\epsilon]]]=0$ for all $\epsilon\in]-\tilde\epsilon_M,\tilde\epsilon_M[$ (cf.~Proposition~\ref{Mr}.) In particular, for $\epsilon\in]-\tilde\epsilon_M,0[$ we have
\begin{eqnarray}\label{110304_eqn2}
\lefteqn{\Lambda^i_1[\epsilon,B^i[\epsilon],B^o[\epsilon],\tilde M_1[\epsilon,B^i[\epsilon],B^o[\epsilon]]](x)}\\
\nonumber
&&
=\int_{\partial\Omega^i}S_n(x -  y)\tilde M^i_1[\epsilon,B^i[\epsilon],B^o[\epsilon]](y)\,d\sigma_y\\
\nonumber
&&
\
+\int_{\partial\Omega^o}S_n(\epsilon x - y)\tilde M^o_1[\epsilon,B^i[\epsilon],B^o[\epsilon]](y)\,d\sigma_y-B^i[\epsilon](x)=0\quad\forall x\in\partial\Omega^i\,.
\end{eqnarray} 
Then, by \eqref{110304_eqn1} and \eqref{110304_eqn2}  we deduce the validity of \eqref{110802_eqn1} in case (iii) with $\tilde\epsilon_*\equiv\tilde\epsilon_M$.

We now assume that (iv) holds. Then there exists $\tilde\epsilon_m\in]0,\inf\{\epsilon_m,\tilde\epsilon\}]$ such that $(B^i[\epsilon],B^o[\epsilon])\in\mathcal{U}$ for all $\epsilon\in]-\tilde\epsilon_m,\tilde\epsilon_m[$. We set
\begin{eqnarray}
\lefteqn{\bar v_\epsilon(x)\equiv \epsilon^{n-2}(\mathrm{sgn} \, \epsilon) \int_{\partial\Omega^i}S_n(x - \epsilon y)\tilde M^i_1[\epsilon,B^i[\epsilon],B^o[\epsilon]](y)\,d\sigma_y}\nonumber\\
&&
\qquad
+\int_{\partial\Omega^o}S_n(x - y)\tilde M^o_1[\epsilon,B^i[\epsilon],B^o[\epsilon]](y)\,d\sigma_y\qquad\forall x\in\mathrm{cl}\Omega(\epsilon)\nonumber
\end{eqnarray} for all $\epsilon\in ]-\tilde\epsilon_m,\tilde\epsilon_m[\setminus\{0\}$ 
and
\[
\bar{v}_0(x)\equiv \int_{\partial\Omega^o}S_n(x - y)\tilde M^o_1[0,B^i[0],B^o[0]](y)\,d\sigma_y\qquad\forall x\in\mathrm{cl}\Omega^o\, .
\]
By Propositions~\ref{L} and \ref{L0} we deduce that $\bar v_\epsilon=v_\epsilon$ for all $\epsilon\in]0,\tilde\epsilon_m[$. So that 
\[
\bar v_{\epsilon}(\epsilon x)=v_{\epsilon}(\epsilon x)\qquad \forall x \in \mathrm{cl}\Omega_m\,, \quad \epsilon\in]0,\tilde\epsilon_m[\, .
\] 
Then we set
\begin{eqnarray*}
\lefteqn{\bar V_m[\epsilon](x)\equiv \int_{\partial\Omega^i}S_n(x -y)\tilde M^i_1[\epsilon,B^i[\epsilon],B^o[\epsilon]](y)\,d\sigma_y}\\
\nonumber
&&
\qquad
+\int_{\partial\Omega^o}S_n(\epsilon x - y)\tilde M^o_1[\epsilon,B^i[\epsilon],B^o[\epsilon]](y)\,d\sigma_y\qquad\forall x\in\mathrm{cl}\Omega_m
\end{eqnarray*}
for all $\epsilon \in ]-\tilde{\epsilon}_m,\tilde{\epsilon}_m[$. We observe that $\bar V_m$ is a real analytic map from $]-\tilde\epsilon_m,\tilde\epsilon_m[$ to $C^{1,\alpha}(\mathrm{cl}\Omega_m)$ and that $\bar{v}_\epsilon(\epsilon x)=\bar{V}_m[\epsilon](x)$ for all $x \in \mathrm{cl}\Omega_m$ and for all $\epsilon \in ]-\tilde{\epsilon}_m,\tilde{\epsilon}_m[\setminus \{0\}$ (see also the argument developed in the proof of Theorem~\ref{even} for $V_m$.)  Then, by the assumption in (iv) and by the Identity Principle for real analytic maps we have $\bar V_m[\epsilon]=V_m[\epsilon]$ for all $\epsilon\in]-\tilde\epsilon_m,\tilde\epsilon_m[$, and thus $\bar v_{\epsilon}(\epsilon \,\cdot\,)_{|\mathrm{cl}\Omega_m}=v_{\epsilon}(\epsilon\, \cdot\,)_{|\mathrm{cl}\Omega_m}$ for all $\epsilon\in]-\tilde\epsilon_m,\tilde\epsilon_m[\setminus \{0\}$. We now note that $\bar v_\epsilon$ is harmonic on $\Omega(\epsilon)$ for all $\epsilon\in]-\tilde\epsilon_m,\tilde\epsilon_m[$. Thus, the equality $\bar v_{\epsilon}(\epsilon \,\cdot\,)_{|\mathrm{cl}\Omega_m}=v_{\epsilon}(\epsilon\, \cdot\,)_{|\mathrm{cl}\Omega_m}$ implies that $\bar v_{\epsilon}=v_{\epsilon}$ on the whole of $\mathrm{cl}\Omega(\epsilon)$ for all $\epsilon\in]-\tilde\epsilon_m,\tilde\epsilon_m[\setminus \{0\}$. In particular, 
\[
\bar v_{\epsilon}(x)=v_{\epsilon}(x)\qquad\forall x\in\partial\Omega^o\,,\ \epsilon\in]-\tilde\epsilon_m,0[\,, 
\] which in turn implies that
\begin{eqnarray}\label{110304_eqn1m}
\lefteqn{-\epsilon^{n-2}\int_{\partial\Omega^i}S_n(x -  \epsilon y)\tilde M^i_1[\epsilon,B^i[\epsilon],B^o[\epsilon]](y)\,d\sigma_y}\\
\nonumber
&&
+\int_{\partial\Omega^o}S_n( x - y)\tilde M^o_1[\epsilon,B^i[\epsilon],B^o[\epsilon]](y)\,d\sigma_y=B^o[\epsilon](x)
\end{eqnarray}  for all $x\in\partial\Omega^o$, $\epsilon\in]-\tilde\epsilon_m,0[$.
By the definition of $\tilde M_1$ in Proposition~\ref{Mr} we have $\Lambda_1[\epsilon,B^i[\epsilon],B^o[\epsilon],\tilde M_1[\epsilon,B^i[\epsilon],B^o[\epsilon]]]=0$ for all $\epsilon\in]-\tilde\epsilon_m,\tilde\epsilon_m[$ (cf.~Proposition~\ref{Mr}.) In particular, for $\epsilon\in]-\tilde\epsilon_m,0[$ we have
\begin{eqnarray}\label{110304_eqn2m}
\lefteqn{\Lambda^o_1[\epsilon,B^i[\epsilon],B^o[\epsilon],\tilde M_1[\epsilon,B^i[\epsilon],B^o[\epsilon]]](x)}\\
\nonumber
&&
=\epsilon^{n-2}\int_{\partial\Omega^i}S_n(x -  \epsilon y)\tilde M^i_1[\epsilon,B^i[\epsilon],B^o[\epsilon]](y)\,d\sigma_y\\
\nonumber
&&
\
+\int_{\partial\Omega^o}S_n( x - y)\tilde M^o_1[\epsilon,B^i[\epsilon],B^o[\epsilon]](y)\,d\sigma_y-B^o[\epsilon](x)=0\quad\forall x\in\partial\Omega^o\,.
\end{eqnarray} 
Then, by \eqref{110304_eqn1m} and \eqref{110304_eqn2m} we deduce that 
\begin{equation}\label{110731_eqn1}
\int_{\partial\Omega^i}S_n(x -  \epsilon y)\tilde M^i_1[\epsilon,B^i[\epsilon],B^o[\epsilon]](y)\,d\sigma_y=0\quad\forall x\in\partial\Omega^o\,,\ \epsilon\in]-\tilde\epsilon_m,0[\,.
\end{equation}
Now let $\epsilon \in ]-\tilde\epsilon_m,0[$. Let $v^\#_\epsilon$ be the function from $\mathbb{R}^n\setminus\epsilon\Omega^i$ to $\mathbb{R}$ defined by
\[
v^\#_\epsilon(x)\equiv  \int_{\partial \Omega^i}S_n(x-\epsilon y)\tilde M^i_1[\epsilon,B^i[\epsilon],B^o[\epsilon]](y)\,d\sigma_y\qquad\forall x\in\mathbb{R}^n\setminus\epsilon\Omega^i\,.
\]
Then we have $\Delta v^\#_\epsilon(x)=0$ for all $x\in\mathbb{R}^n\setminus\mathrm{cl}\Omega^o$ and equality \eqref{110731_eqn1} implies that $v^\#_{\epsilon}(x)=0$ for all $x\in\partial\Omega^o$. Moreover, by the decay properties at infinity of $S_n$ we have $\lim_{|x|\to \infty}v^\#_\epsilon(x)=0$. Thus $v^\#_{\epsilon|\mathbb{R}^n\setminus\mathrm{cl}\Omega^o}$ coincides with the unique solution of the exterior homogeneous Dirichlet  problem in $\mathbb{R}^n\setminus\mathrm{cl}\Omega^o$ which vanishes at infinity.  Accordingly $v^\#_\epsilon(x)=0$ for all $x\in\mathbb{R}^n\setminus\mathrm{cl}\Omega^o$. We now observe that $\Delta v^\#_\epsilon(x)=0$ for all $x\in\mathbb{R}^n\setminus\epsilon\mathrm{cl}\Omega^i$. Thus, by the Identity Principle for real analytic functions we have $v^\#_\epsilon(x)=0$ for all $x\in\mathbb{R}^n \setminus \epsilon \Omega^i$. In particular, $v^\#_\epsilon(\epsilon x)=0$ for all $x\in\partial\Omega^i$. Then by a straightforward calculation we deduce the validity of \eqref{110802_eqn1} in case (iv) with $\tilde\epsilon_*\equiv\tilde\epsilon_m$.

Hence, the equality in \eqref{110802_eqn1} holds both in case (iii) and (iv) with $\tilde\epsilon_*\in]0,\tilde\epsilon]$. Then Lemma~\ref{homeo} implies that $\tilde M^i_1[\epsilon,B^i[\epsilon],B^o[\epsilon]]=0$ for all $\epsilon\in]-\tilde \epsilon_*,0[$. Thus,  by a standard argument based on the Identity Principle for real analytic functions we deduce that
\begin{equation}\label{110802_eqn2}
\tilde M^i_1[\epsilon,B^i[\epsilon],B^o[\epsilon]]=0\qquad\forall \epsilon\in]-\tilde \epsilon_*,\tilde \epsilon_*[\,.
\end{equation}

We now observe that the equality in \eqref{110802_eqn2} implies that 
\begin{equation}\label{110606_eqn3}
\Lambda_\theta[\epsilon,B^i[\epsilon],B^o[\epsilon],0,\tilde M^o_1[\epsilon,B^i[\epsilon],B^o[\epsilon]]]=(0,0)\quad\forall  \epsilon\in]-\tilde \epsilon_*,\tilde \epsilon_*[\,,\ \theta\in\{-1,1\}\,.
\end{equation}
Let $M_+$ and $M_-$ be as in Proposition~\ref{Modd}. Then by equality \eqref{110606_eqn3}, and by Lemma~\ref{homeo}, and by  Propositions~\ref{L}, \ref{Modd}, and by a standard argument based on the Identity Principle for real analytic functions we verify that $M_+^i[\epsilon,B^i[\epsilon],B^o[\epsilon]]=0$ for all $\epsilon\in]0,\epsilon_1[$, and that $M_-^i[\epsilon,B^i[\epsilon],B^o[\epsilon]]=0$ for all $\epsilon\in]-\epsilon_1,0[$, and that $M_+^o[\epsilon,B^i[\epsilon],B^o[\epsilon]]=\tilde M_1^o[\epsilon,B^i[\epsilon],B^o[\epsilon]]$ for $\epsilon\in]0,\tilde\epsilon_*[$, and that $M_-^o[\epsilon,B^i[\epsilon],B^o[\epsilon]]=\tilde M_1^o[\epsilon,B^i[\epsilon],B^o[\epsilon]]$ for $\epsilon\in]-\tilde\epsilon_*,0[$. So, if we set
\[
m^o[\epsilon]\equiv
\left\{
\begin{array}{ll}
M^o_+[\epsilon,B^i[\epsilon],B^o[\epsilon]]&\text{if }\epsilon\in[\tilde\epsilon_*,\epsilon_1[\,,\\
\tilde M^o_1[\epsilon,B^i[\epsilon],B^o[\epsilon]]&\text{if }\epsilon\in]-\tilde\epsilon_*,\tilde\epsilon_*[\,,\\
M^o_-[\epsilon,B^i[\epsilon],B^o[\epsilon]]&\text{if }\epsilon\in]-\epsilon_1,-\tilde\epsilon_*]\,
\end{array}
\right.
\] and we define
\[
w_\epsilon(x)\equiv \int_{\partial\Omega^o}S_n(x - y)m^o[\epsilon](y)\,d\sigma_y\quad\forall x\in\mathrm{cl}\Omega^o\,,\ \epsilon\in]-\epsilon_1,\epsilon_1[\,,
\] then $\{w_\epsilon\}_{\epsilon\in]-\epsilon_1,\epsilon_1[}$ satisfies the conditions in (c0), (c1) and $v_\epsilon=w_{\epsilon|\mathrm{cl}\Omega(\epsilon)}$ for all $\epsilon\in]-\epsilon_1,\epsilon_1[$ (see also Propositions~\ref{L} and \ref{L0}.) The validity of the Theorem is now verified.
\end{proof}

\medskip

We now show that in Theorem~\ref{odd} it is necessary to require the validity of condition (iii) or of condition (iv). To do so, we construct for $n$ odd a family of functions $\{v_\epsilon\}_{\epsilon\in]-\epsilon_1,\epsilon_1[}$ which satisfies the conditions in (b0), (i), (ii) but not the conditions in (iii) and (iv) (see Example~\ref{ex1} here below.) In particular, for such a family it is not possible to find  $\{w_\epsilon\}_{\epsilon\in]-\epsilon_1,\epsilon_1[}$ which satisfies the conditions in (c0), (c1) and such that $v_\epsilon=w_{\epsilon|\mathrm{cl}\Omega(\epsilon)}$ for all $\epsilon\in]-\epsilon_1,\epsilon_1[$.

\begin{exam}\label{ex1}
Assume that the dimension $n$ is odd. Assume that $\Omega^o$ and $\Omega^i$ coincide with the set $\{x\in\mathbb{R}^n\,:\,|x|<1\}$. Let  $\Omega(\epsilon)\equiv\{x\in\mathbb{R}^n\,:\,|\epsilon|<|x|<1\}$ for all $\epsilon\in]-1,1[$. Let $v_\epsilon$ be the function from $\mathrm{cl}\Omega(\epsilon)$ to $\mathbb{R}$ defined by
\[
v_\epsilon(x)\equiv \frac{\epsilon\,|\epsilon|^{n-2}}{1-|\epsilon|^{n-2}}\left(|x|^{2-n}-1\right)\qquad\forall x\in\mathrm{cl}\Omega(\epsilon)
\]
for all $\epsilon\in]-1,1[\setminus \{0\}$. Let $v_0(x)\equiv 0$ for all $x \in \mathrm{cl}\Omega^o$. 
Then $\{v_\epsilon\}_{\epsilon\in]-1,1[}$ satisfies the condition in (b0) and the conditions in (i), (ii) of Theorem~\ref{odd} but not the conditions in (iii) and (iv).
\end{exam}
\begin{proof} Let $\epsilon\in]-1,1[\setminus\{0\}$. Then $v_\epsilon\in C^{1,\alpha}(\mathrm{cl}\Omega(\epsilon))$ and we have $\Delta v_\epsilon=0$ in $\Omega(\epsilon)$. Further $v_{\epsilon}(x)=0$ if $|x|=1$, and $v_{\epsilon}(x)=\epsilon$ if $|x|=|\epsilon|$ for all $\epsilon\in]-1,1[$. Thus $\{v_\epsilon\}_{\epsilon\in]-1,1[}$ satisfies the condition in (b0) and the conditions in (i), (ii) of Theorem~\ref{odd}. Now let $x_0$ be a point of $\mathbb{R}^n$ with $0<|x_0|<1$. We show that the map which takes $\epsilon$ to $v_\epsilon(x_0)$ is not real analytic in a neighborhood of $\epsilon=0$. In particular $\{v_\epsilon\}_{\epsilon\in]-1,1[}$ does not satisfy the condition in (iii).  To do so, we prove that the map which takes $\epsilon$ to ${\epsilon\,|\epsilon|^{n-2}}/({1-|\epsilon|^{n-2}})$ is not in $C^{n-1}$ for $\epsilon$  in a neighborhood of $0$. We note that
\[
\frac{\epsilon\,|\epsilon|^{n-2}}{1-|\epsilon|^{n-2}}={\epsilon\,|\epsilon|^{n-2}}\,\psi_1(\epsilon)+\psi_2(\epsilon)\qquad\forall\epsilon\in]-1,1[
\] with $\psi_1(\epsilon)\equiv ({1-\epsilon^{2(n-2)}})^{-1}$ and $\psi_2(\epsilon)\equiv \epsilon^{2(n-2)+1}({1-\epsilon^{2(n-2)}})^{-1}$. The maps $\psi_1$ and $\psi_2$ are real analytic from $]-1,1[$ to $\mathbb{R}$ and we have $\psi_1(0)=1$. We observe that $(\frac{d}{d\epsilon})^{(n-1)}(\epsilon|\epsilon|^{n-2})=(n-1)!\,\mathrm{sgn}\,\epsilon$. Then we deduce that  
\begin{equation}\label{ex1_eqn1}
\biggl(\frac{d}{d\epsilon}\biggr)^{(n-1)}\left(\frac{\epsilon\,|\epsilon|^{n-2}}{1-|\epsilon|^{n-2}}\right)={(n-1)!\,(\mathrm{sgn}\,{\epsilon})\,\psi_1(\epsilon)}+\psi_3(\epsilon)\qquad\forall\epsilon\in]-1,1[\setminus\{0\}\,,
\end{equation} where $\psi_3$ is a continuous map from $]-1,1[$ to $\mathbb{R}$. The function on the right hand side of \eqref{ex1_eqn1} has no continuous extension on $]-1,1[$ and our proof is complete. The proof that $\{v_\epsilon\}_{\epsilon\in]-1,1[}$ does not satisfy (iv) is similar and is accordingly omitted. 
\end{proof}

\medskip

We show in the following Example~\ref{ex2}  that analogs of Theorem~\ref{even} and statement (j) do not hold if we replace the assumption that $u_{\epsilon}$ is harmonic on $\Omega(\epsilon)$ for $\epsilon\in]0,\epsilon_1[$ with the weaker assumption that $u_{\epsilon}$ is real analytic on $\Omega(\epsilon)$. Similarly, we show in Example~\ref{ex3} that analogs of Theorem~\ref{odd} and statement (jj) are not true if we replace the assumption that $v_{\epsilon}$ and $w_\epsilon$ are harmonic on $\Omega(\epsilon)$ and $\Omega^o$, respectively, with the weaker assumption that $v_{\epsilon}$ and $w_\epsilon$ are real analytic on $\Omega(\epsilon)$ and $\Omega^o$, respectively. 

\begin{exam}\label{ex2} Let $\Omega^o$ and $\Omega^i$ be equal to $\{x\in\mathbb{R}^n\,:\,|x|<1\}$. Let  $\Omega(\epsilon)\equiv\{x\in\mathbb{R}^n\,:\,|\epsilon|<|x|<1\}$ for all $\epsilon\in]-1,1[$. Let $u_\epsilon$ be the function of $C^{1,\alpha}(\mathrm{cl}\Omega(\epsilon))$ defined by 
\[
u_\epsilon(x)\equiv|x|\qquad\forall x\in\mathrm{cl}\Omega(\epsilon)\,,\ \epsilon\in]0,1[\,.
\] Then $u_\epsilon$ is real analytic on $\Omega(\epsilon)$ for all $\epsilon\in]0,1[$ and the family 
$\{u_\epsilon\}_{\epsilon\in]0,1[}$ satisfies the conditions in (a1), (a2), but there exists no family of functions $\{v_\epsilon\}_{\epsilon\in]-1,1[}$ on $\Omega(\epsilon)$ which satisfies the conditions in (b1), (b2) and such that $u_\epsilon=v_\epsilon$ for all $\epsilon\in]0,1[$.
\end{exam} 
\begin{proof} Clearly $u_\epsilon$ belongs to $C^{1,\alpha}(\mathrm{cl}\Omega(\epsilon))$ and is real analytic on $\Omega(\epsilon)$ for all $\epsilon\in]0,1[$. Moreover, a straightforward calculation shows that $\{u_\epsilon\}_{\epsilon\in]0,1[}$ satisfies the conditions in (a1), (a2). Assume by contradiction that there exists a family $\{v_\epsilon\}_{\epsilon\in]-1,1[}$ of functions on $\Omega(\epsilon)$ which satisfies the conditions in (b1), (b2) and such that $u_\epsilon=v_\epsilon$ for all $\epsilon\in]0,1[$. Then condition (b1) and the Identity Principle for real analytic maps imply that we have $v_{\epsilon}(x)=|x|$ for all $x\in\mathbb{R}^n$ with $1/2\leq |x|\le 1$ and for all $\epsilon\in]-1/2,1/2[$. Condition (b2) and the Identity Principle for real analytic maps imply that we have $v_{\epsilon}(\epsilon x)=\epsilon|x|$ for all $x\in\mathbb{R}^n$ with $1\le|x|\leq 2$ and for all $\epsilon\in]-1/2,1/2[\setminus\{0\}$. Let $\epsilon^*\in]-1/2,-1/4[$. Let $x^*\in\mathbb{R}^n$ with $1/2<|x^*|<2|\epsilon^*|$. So that $1<|x^*/\epsilon^*|<2$. Then $v_{\epsilon^*}(x^*)=|x^*|$ and $v_{\epsilon^*}(x^*)=v_{\epsilon^*}(\epsilon^*(x^*/\epsilon^*))=\epsilon^*|x^*/\epsilon^*|=-|x^*|$. A contradiction. 
\end{proof}

\begin{exam}\label{ex3} Let $\Omega^o$ and $\Omega^i$ be equal to $\{x\in\mathbb{R}^n\,:\,|x|<1\}$. Let  $\Omega(\epsilon)\equiv\{x\in\mathbb{R}^n\,:\,|\epsilon|<|x|<1\}$ for all $\epsilon\in]-1,1[$. Let $v_\epsilon$ be the function of $C^{1,\alpha}(\mathrm{cl}\Omega(\epsilon))$ defined by 
\[
v_\epsilon(x)\equiv\epsilon^2/|x|^2
\qquad\forall x\in\mathrm{cl}\Omega(\epsilon)\,,\ \epsilon\in]-1,1[\setminus\{0\}\,.
\] Let $v_0(x)\equiv 0$ for all $x\in\mathrm{cl}\Omega^o$. Then $v_0$ is real analytic on $\Omega^o$, and $v_\epsilon$ is real analytic on $\Omega(\epsilon)$ for all $\epsilon\in]-1,1[\setminus\{0\}$, and the family 
$\{v_\epsilon\}_{\epsilon\in]-1,1[}$ satisfies the conditions in (b1), (b2), but for any fixed $\epsilon^*\in]-1,1[\setminus\{0\}$ there exists no function $w_{\epsilon^*}$ real analytic on $\Omega^o$ which satisfies the equality $v_{\epsilon^*}=w_{\epsilon^*|\mathrm{cl}\Omega(\epsilon^*)}$.
\end{exam} 
\begin{proof} Clearly $v_\epsilon$ belongs to $C^{1,\alpha}(\mathrm{cl}\Omega(\epsilon))$ and is real analytic on $\Omega(\epsilon)$ for all $\epsilon\in]-1,1[\setminus\{0\}$. Moreover, a straightforward calculation that $\{v_\epsilon\}_{\epsilon\in]-1,1[}$ satisfies the conditions in (b1), (b2). Now let $\epsilon^*\in]-1,1[\setminus\{0\}$. Let $\tilde w_{\epsilon^*}$ be a real analytic map on $ \Omega^o\setminus\{0\}$ such that $v_{\epsilon^*}=\tilde w_{\epsilon^*|\mathrm{cl}\Omega(\epsilon^*)}$. By the Identity Principle for real analytic maps we deduce that $\tilde w_{\epsilon^*}(x)=(\epsilon^*)^2/|x|^2$ for all $x\in\Omega^o\setminus\{0\}$. Thus $\tilde w_{\epsilon^*}$ has no continuous extension on $ \Omega^o$ and the validity of the statement follows.
\end{proof}

\section{Some particular cases}\label{appl}

In this section we consider   some particular cases and we show some consequences of Theorems~\ref{even} and \ref{odd}. In the following Proposition~\ref{right} we show that the family $\{u_\epsilon\}_{\epsilon\in]0,\epsilon_0[}$ of the solutions of the boundary value problem in \eqref{dir1} satisfies the conditions in (a1) and (a2) for some $\epsilon_1\in]0,\epsilon_0]$.

\begin{prop}\label{right}
Let $\Omega^i$, $\Omega^o$ be as in \eqref{e1}. Let $\epsilon_0$ be as in \eqref{e2}. Let $(f^i,f^o)\in C^{1,\alpha}(\partial\Omega^i)\times C^{1,\alpha}(\partial\Omega^o)$. Let $u_\epsilon$ denote the unique solution in $C^{1,\alpha}(\mathrm{cl}\Omega(\epsilon))$ of the boundary value problem in \eqref{dir1} for all $\epsilon\in]0,\epsilon_0[$. Then there exists $\epsilon_1\in]0,\epsilon_0]$ such that the family $\{u_\epsilon\}_{\epsilon\in]0,\epsilon_0[}$ satisfies the conditions in (a1) and (a2). If the dimension $n$ is even, then we can take $\epsilon_1=\epsilon_0$.
\end{prop} 
\begin{proof} If the dimension $n$ is even, then the validity of the Proposition follows by Theorem~\ref{even} with $\epsilon_1\equiv\epsilon_0$ and $B^i[\epsilon]\equiv f^i$, $B^o[\epsilon]\equiv f^o$  for all $\epsilon\in]-\epsilon_1,\epsilon_1[$. So let $n$ be odd.  Let $\tilde M_1\equiv(\tilde M^i_1,\tilde M^o_1)$, $\tilde\epsilon$, $\mathcal{U}$ be as in Propositions~\ref{Mr} with $\theta\equiv 1$, $\tilde f^i\equiv f^i$ and $\tilde f^o\equiv f^o$. We set $\epsilon_1\equiv \tilde\epsilon$.  Let $\Omega_M$ and $\epsilon_M$ be as in (a1). Let $\tilde v_\epsilon$ be defined as in \eqref{110304_eqn3}, \eqref{110606_eqn1} with $B^i[\epsilon]\equiv f^i$ and $B^o[\epsilon]\equiv f^o$ for all $\epsilon\in]-\tilde\epsilon,\tilde\epsilon[$. Then we set $U_M[\epsilon]\equiv\tilde v_{\epsilon|\mathrm{cl}\Omega_M}$ for all $\epsilon\in]-\epsilon_M,\epsilon_M[$. Then we show that $U_M$ is real analytic from $]-\epsilon_M,\epsilon_M[$ to $C^{1,\alpha}(\mathrm{cl}\Omega_M)$ (see also the argument exploited in the proof of Theorem~\ref{even} for $V_M$.)   The validity of (a1) is thus proved. Now let $\Omega_m$ and $\epsilon_m$ be as in (a2). Let $U_m[\epsilon]$ be defined by
\begin{eqnarray*}
\lefteqn{U_m[\epsilon](x)\equiv \int_{\partial\Omega^i}S_n(x -  y)\tilde M^i_1[\epsilon,f^i,f^o](y)\,d\sigma_y}\\
&&
\qquad\quad 
+\int_{\partial\Omega^o}S_n(\epsilon x - y)\tilde M^o_1[\epsilon,f^i,f^o](y)\,d\sigma_y\qquad\forall x\in\mathrm{cl}\Omega_m
\end{eqnarray*} for all $\epsilon\in]-\epsilon_m,\epsilon_m[$. Clearly,
\[
u_\epsilon(\epsilon x)=U_m[\epsilon](x) \qquad \forall x \in \mathrm{cl}\Omega_m
\]
for all $\epsilon \in ]0,\epsilon_m[$. We verify that $U_m$ is real analytic from $]-\epsilon_m,\epsilon_m[$ to $C^{1,\alpha}(\mathrm{cl}\Omega_m)$ (see also the argument exploited in the proof of Theorem~\ref{even} for $V_m$.)  Accordingly the validity of (a2) follows.
 \end{proof}\medskip

In the following Proposition~\ref{sym} we assume that $n$ is even and we consider a family $\{u_\epsilon\}_{\epsilon\in]0,\epsilon_1[}$ of harmonic functions on $\Omega(\epsilon)$ which satisfies the conditions in Theorem~\ref{even}. Then we investigate the power series that describe $u_{\epsilon|\mathrm{cl}\Omega_M}$ and $u_\epsilon(\epsilon\,\cdot\,)_{|\mathrm{cl}\Omega_m}$ for $\epsilon$ small and positive under suitable symmetry assumptions on $B^i$, $B^o$, $\Omega^i$ and $\Omega^o$.

\begin{prop}\label{sym}
Assume that $n$ is even. Let $\Omega^i$, $\Omega^o$ be as in \eqref{e1}.  Let $\epsilon_0$ be as in \eqref{e2}.  Let $\epsilon_1\in]0,\epsilon_0]$. Let $\{u_\epsilon\}_{\epsilon\in]0,\epsilon_1[}$, $B^i$ and $B^o$ be as in Theorem~\ref{even}.    Let $\Omega_M$, $\epsilon_M$ be as in (b1). Let $\Omega_m$, $\epsilon_m$ be as in (b2). Let $\zeta\in\{-1,1\}$.  Then the following statements hold.  
\begin{enumerate}
\item[(i)] If $\Omega^i=-\Omega^i$ and \[
B^i[\epsilon](x)=\zeta B^i[-\epsilon](-x)\,,\quad B^o[\epsilon](y)=\zeta B^o[-\epsilon](y)
\] for all $x\in\partial\Omega^i$, $y\in\partial\Omega^o$, $\epsilon\in]-\epsilon_1,\epsilon_1[$, then there exist $\tilde\epsilon_M\in]0,\epsilon_M[$ and a sequence  $\{u_{M,j}\}_{j\in\mathbb{N}}$ in $C^{1,\alpha}(\mathrm{cl}\Omega_M)$ such that 
\[
u_{\epsilon|\mathrm{cl}\Omega_M}=\epsilon^{(1-\zeta)/2}\sum_{j=0}^\infty u_{M,j}\,\epsilon^{2j}\qquad\forall \epsilon\in]0,\tilde\epsilon_M[\,,
\] where the series converges in $C^{1,\alpha}(\mathrm{cl}\Omega_M)$.
\item[(ii)] If $\Omega^o=-\Omega^o$ and \[
B^i[\epsilon](x)=\zeta B^i[-\epsilon](x)\,,\quad B^o[\epsilon](y)=\zeta B^o[-\epsilon](-y)
\] for all $x\in\partial\Omega^i$, $y\in\partial\Omega^o$, $\epsilon\in]-\epsilon_1,\epsilon_1[$, then there exist $\tilde\epsilon_m\in]0,\epsilon_m[$ and  a sequence   $\{u_{m,j}\}_{j\in\mathbb{N}}$ in $C^{1,\alpha}(\mathrm{cl}\Omega_m)$ such that 
\[
u_{\epsilon}(\epsilon\,\cdot\,)_{|\mathrm{cl}\Omega_m}=\epsilon^{(1-\zeta)/2}\sum_{j=0}^\infty u_{m,j}\,\epsilon^{2j}\qquad\forall \epsilon\in]0,\tilde\epsilon_m[\,,
\] where the series converges in  $C^{1,\alpha}(\mathrm{cl}\Omega_m)$.
\end{enumerate}
\end{prop} 
\begin{proof}
Let $\{v_\epsilon\}_{\epsilon\in]-\epsilon_1,\epsilon_1[}$ be as in Theorem~\ref{even}. Then $\{v_\epsilon\}_{\epsilon\in]-\epsilon_1,\epsilon_1[}$ satisfies the conditions in (b1), (b2) and we deduce that there exist  $\tilde\epsilon_M\in]0,\epsilon_M[$, $\tilde\epsilon_m\in]0,\epsilon_m[$ and sequences  $\{v_{M,j}\}_{j\in\mathbb{N}}$ in $C^{1,\alpha}(\mathrm{cl}\Omega_M)$ and $\{v_{m,j}\}_{j\in\mathbb{N}}$ in $C^{1,\alpha}(\mathrm{cl}\Omega_m)$ such that 
\begin{eqnarray*}
v_{\epsilon|\mathrm{cl}\Omega_M}=&\sum_{j=0}^\infty v_{M,j}\,\epsilon^{j}& \forall \epsilon\in]-\tilde\epsilon_M,\tilde\epsilon_M[\,,\\ v_{\epsilon}(\epsilon\,\cdot\,)_{|\mathrm{cl}\Omega_m}=&\sum_{j=0}^\infty v_{m,j}\,\epsilon^{j}& \forall \epsilon\in]-\tilde\epsilon_m,\tilde\epsilon_m[\setminus\{0\}\,,
\end{eqnarray*} where the  first and second series converge in $C^{1,\alpha}(\mathrm{cl}\Omega_M)$ and $C^{1,\alpha}(\mathrm{cl}\Omega_m)$, respectively. Then, by the assumptions in (i) and by Proposition~\ref{L}, and by the uniqueness of the solution of the Dirichlet problem in $\Omega(\epsilon)$ for all $\epsilon\in]-\tilde\epsilon_M,\tilde\epsilon_M[\setminus\{0\}$,  we deduce that $\Omega(\epsilon)=\Omega(-\epsilon)$ and that $v_\epsilon=\zeta v_{-\epsilon}$ for all $\epsilon\in]-\tilde\epsilon_M,\tilde\epsilon_M[\setminus\{0\}$. Thus we have
$\sum_{j=0}^\infty v_{M,j}(-\epsilon)^{j}=\zeta \sum_{j=0}^\infty v_{M,j}\epsilon^{j}$ for all $\epsilon\in]-\tilde\epsilon_M,\tilde\epsilon_M[$, which implies that $v_{M,2j+(1+\zeta)/2}=0$ for all $j\in\mathbb{N}$.  If we now set $u_{M,j}\equiv v_{M,2j+(1-\zeta)/2}$ for all $j\in\mathbb{N}$, then the validity of statement (i) follows. Similarly, by the assumptions in (ii) and by Proposition~\ref{L}, and by the uniqueness of the solution of the Dirichlet problem in $\Omega(\epsilon)$ for $\epsilon\in]-\tilde\epsilon_m,\tilde\epsilon_m[\setminus\{0\}$,  we deduce that $\Omega(\epsilon)=-\Omega(-\epsilon)$ and that $v_\epsilon(x)=\zeta v_{-\epsilon}(-x)$ for all $x\in\mathrm{cl}\Omega(\epsilon)$ and all $\epsilon\in]-\tilde\epsilon_m,\tilde\epsilon_m[\setminus\{0\}$. In particular $v_\epsilon(\epsilon x)=\zeta v_{-\epsilon}(-\epsilon x)$ for all $x\in\mathrm{cl}\Omega_m$ and all $\epsilon\in]-\tilde\epsilon_m,\tilde\epsilon_m[\setminus\{0\}$. We deduce that
$\sum_{j=0}^\infty v_{m,j}(-\epsilon)^{j}=\zeta \sum_{j=0}^\infty v_{m,j}\epsilon^{j}$ for all $\epsilon\in]-\tilde\epsilon_m,\tilde\epsilon_m[$, which in turn implies that $v_{m,2j+(1+\zeta)/2}=0$ for all $j\in\mathbb{N}$.  If we now set $u_{m,j}\equiv v_{m,2j+(1-\zeta)/2}$ for all $j\in\mathbb{N}$, then the validity of statement (ii) follows. 
\end{proof}\medskip

Now let $n$ be odd. Let $\{u_\epsilon\}_{\epsilon\in]0,\epsilon_0[}$ denote the family of the solutions of \eqref{dir1}. As an immediate consequence of the following Proposition~\ref{const} one can verify that the equalities in \eqref{uUM} and \eqref{uUm} hold for $\epsilon$ negative only if there exists $c\in\mathbb{R}$ such that $u_\epsilon(x)=c$ for all $x\in\mathrm{cl}\Omega(\epsilon)$ and $\epsilon\in]0,\epsilon_0[$.

\begin{prop}\label{const}
Assume that $n$ is odd. Let $\Omega^i$, $\Omega^o$ be as in \eqref{e1}. Let $\epsilon_0$ be as in \eqref{e2}. Let $\epsilon_1\in]0,\epsilon_0]$. Let $\{v_\epsilon\}_{\epsilon\in]-\epsilon_1,\epsilon_1[}$, $B^i$ and $B^o$ be as in Theorem~\ref{odd}. Then the following statements are equivalent.
\begin{enumerate}
\item[(i)] There exist functions $f^i\in C^{1,\alpha}(\partial\Omega^i)$ and $f^o\in C^{1,\alpha}(\partial\Omega^o)$ such that $B^i[\epsilon]=f^i$ and $B^o[\epsilon]=f^o$ for all $\epsilon\in]-\epsilon_1,\epsilon_1[$.
\item[(ii)] There exists a constant $c\in\mathbb{R}$ such that $v_\epsilon(x)=c$ for all $x\in\mathrm{cl}\Omega(\epsilon)$ and all $\epsilon\in]-\epsilon_1,\epsilon_1[$.
\end{enumerate}  
\end{prop}
\begin{proof} Clearly statement (ii) implies (i). So we have to show that  (i) implies (ii). By Theorem~\ref{odd} there exists a family $\{w_\epsilon\}_{\epsilon\in]-\epsilon_1,\epsilon_1[}$ of harmonic functions on $\Omega^o$ such that $v_\epsilon=w_{\epsilon|\mathrm{cl}\Omega(\epsilon)}$ for all $\epsilon\in]-\epsilon_1,\epsilon_1[$. In particular we have $w_{\epsilon|\partial\Omega^o}=B^o[\epsilon]=f^o$ for all $\epsilon\in]-\epsilon_1,\epsilon_1[$ and $w_{\epsilon}(\epsilon\,\cdot\,)_{|\partial\Omega^i}=B^i[\epsilon]=f^i$ for all $\epsilon\in]-\epsilon_1,\epsilon_1[\setminus\{0\}$. By the uniqueness of the solution of the Dirichlet problem in $\Omega^o$ and by Lemma~\ref{homeo} we deduce that $w_\epsilon=w_0$ for all $\epsilon\in]-\epsilon_1,\epsilon_1[$ and that there exists $\mu^o \in C^{0,\alpha}(\partial \Omega^o)$ such that
\begin{equation}\label{110915eq1}
w_\epsilon (x)=\int_{\partial\Omega^o}S_n(x - y)\mu^o(y)\,d\sigma_y\quad \forall x\in\mathrm{cl}\Omega^o\,,\ \epsilon \in ]-\epsilon_1,\epsilon_1[\,.
\end{equation} We now prove that $f^i$ is constant on  $\partial\Omega^i$. Indeed,  equality $w_{\epsilon}(\epsilon\,\cdot\,)_{|\partial\Omega^i}=f^i$ for all $\epsilon\in]-\epsilon_1,\epsilon_1[\setminus\{0\}$ and \eqref{110915eq1} imply that
\begin{equation}\label{0721eqn2}
f^i(x)=\int_{\partial\Omega^o}S_n(\epsilon x - y)\mu^o(y)\,d\sigma_y\quad \forall x\in\partial \Omega^i\,,\ \epsilon \in ]-\epsilon_1,\epsilon_1[\setminus \{0\}\,.
\end{equation}
Since the map from $]-\epsilon_1,\epsilon_1[$ to $C^{1,\alpha}(\partial \Omega^i)$ which takes $\epsilon$ to the function $\int_{\partial\Omega^o}S_n(\epsilon x - y)\mu^o(y)\,d\sigma_y$ of $ x\in\partial \Omega^i$ is real analytic, we can take the limit as $\epsilon\to 0$ in \eqref{0721eqn2} and we obtain  
\[
f^i(x)=\int_{\partial\Omega^o}S_n(y)\mu^o(y)\,d\sigma_y=w_0(0)\quad \forall x\in\partial \Omega^i
\] (cf.~Lemma~\ref{anal} (i).)
Now let $\epsilon^*\in]0,\epsilon_1[$ be fixed. Then we have $w_{0}(x)=w_{\epsilon^*}(x)=f^i(x/\epsilon^*)=w_0(0)$ for all $x\in\epsilon^*\partial\Omega^i$. Since $w_{0}$ is harmonic in $\epsilon^*\Omega^i$ we deduce that $w_{0}(x)=w_0(0)$ for all $x\in\epsilon^*\mathrm{cl}\Omega^i$. Then, by the Identity Principle for real analytic functions $w_{0}(x)=w_0(0)$ for all $x\in\mathrm{cl}\Omega^o$. By defining   $c\equiv w_0(0)$ the validity of statement (ii) follows. \end{proof}

\section{Acknowledgments}

The authors wish to thank their teacher Prof.~M.~Lanza de Cristoforis for his precious help.


\begin{thebibliography}{12}
\expandafter\ifx\csname natexlab\endcsname\relax\def\natexlab#1{#1}\fi
\providecommand{\bibinfo}[2]{#2}
\ifx\xfnm\relax \def\xfnm[#1]{\unskip,\space#1}\fi
\bibitem{GiTr01}
\bibinfo{author}{D.~Gilbarg}, \bibinfo{author}{N.~S. Trudinger},
  \bibinfo{title}{Elliptic partial differential equations of second order}, Classics in Mathematics, \bibinfo{publisher}{Springer-Verlag}, \bibinfo{address}{Berlin}, \bibinfo{year}{2001}.
\bibitem{MaNaPl00}
\bibinfo{author}{V.~Maz'ya}, \bibinfo{author}{S.~Nazarov},
  \bibinfo{author}{B.~Plamenevskij}, \bibinfo{title}{Asymptotic theory of
  elliptic boundary value problems in singularly perturbed domains. {V}ols.
  {I}, {II}}, volumes \bibinfo{volume}{111, 112} of \textit{\bibinfo{series}{Operator
  Theory: Advances and Applications}}, \bibinfo{publisher}{Birkh{\"a}user
  Verlag}, \bibinfo{address}{Basel}, \bibinfo{year}{2000}.
\bibitem{La02}
\bibinfo{author}{M.~Lanza~de Cristoforis},
\newblock \bibinfo{title}{Asymptotic behaviour of the conformal representation
  of a {J}ordan domain with a small hole in {S}chauder spaces},
\newblock \bibinfo{journal}{Comput. Methods Funct. Theory} \bibinfo{volume}{2}
  (\bibinfo{year}{2002}) \bibinfo{pages}{1--27}.
\bibitem{De85}
\bibinfo{author}{K.~Deimling}, \bibinfo{title}{Nonlinear functional analysis},
  \bibinfo{publisher}{Springer-Verlag}, \bibinfo{address}{Berlin},
  \bibinfo{year}{1985}.
\bibitem{La08}
\bibinfo{author}{M.~Lanza~de Cristoforis},
\newblock \bibinfo{title}{Asymptotic behavior of the solutions of the
  {D}irichlet problem for the {L}aplace operator in a domain with a small hole.
  {A} functional analytic approach},
\newblock \bibinfo{journal}{Analysis (Munich)} \bibinfo{volume}{28}
  (\bibinfo{year}{2008}) \bibinfo{pages}{63--93}.
\bibitem{Lan07a}
\bibinfo{author}{M.~Lanza~de Cristoforis},
\newblock \bibinfo{title}{Asymptotic behavior of the solutions of a nonlinear
  {R}obin problem for the {L}aplace operator in a domain with a small hole: a
  functional analytic approach},
\newblock \bibinfo{journal}{Complex Var. Elliptic Equ.} \bibinfo{volume}{52}
  (\bibinfo{year}{2007}) \bibinfo{pages}{945--977}.
\bibitem{La10}
\bibinfo{author}{M.~Lanza~de Cristoforis},
\newblock \bibinfo{title}{Asymptotic behaviour of the solutions of a non-linear
  transmission problem for the {L}aplace operator in a domain with a small
  hole. {A} functional analytic approach},
\newblock \bibinfo{journal}{Complex Var. Elliptic Equ.} \bibinfo{volume}{55}
  (\bibinfo{year}{2010}) \bibinfo{pages}{269--303}.
\bibitem{Mi65}
\bibinfo{author}{C.~Miranda},
\newblock \bibinfo{title}{Sulle propriet{\`a} di regolarit{\`a} di certe
  trasformazioni integrali},
\newblock \bibinfo{journal}{Atti Accad. Naz. Lincei Mem. Cl. Sci. Fis. Mat.
  Natur. Sez. I (8)} \bibinfo{volume}{7} (\bibinfo{year}{1965})
  \bibinfo{pages}{303--336}.
\bibitem{LaMu10b}
\bibinfo{author}{M.~Lanza~de Cristoforis}, \bibinfo{author}{P.~Musolino},
  \bibinfo{title}{A real analyticity result for a nonlinear integral
  operator},  \bibinfo{note}{Submitted}.
\bibitem{BT73}
\bibinfo{author}{R.~B{{\"o}}hme}, \bibinfo{author}{F.~Tomi},
\newblock \bibinfo{title}{Zur {S}truktur der {L}{\"o}sungsmenge des
  {P}lateauproblems},
\newblock \bibinfo{journal}{Math. Z.} \bibinfo{volume}{133}
  (\bibinfo{year}{1973}) \bibinfo{pages}{1--29}.
\bibitem{Hen82}
\bibinfo{author}{D.~Henry}, \bibinfo{title}{Topics in Nonlinear Analysis},
  volume \bibinfo{volume}{192} of \textit{\bibinfo{series}{Trabalho de
  Matematica}}, \bibinfo{address}{Universidade de Brasilia},
  \bibinfo{year}{1982}.
\bibitem{Val88}
\bibinfo{author}{T.~Valent}, \bibinfo{title}{Boundary value problems of finite
  elasticity}, volume~\bibinfo{volume}{31} of \textit{\bibinfo{series}{Springer
  Tracts in Natural Philosophy}}, \bibinfo{publisher}{Springer-Verlag},
  \bibinfo{address}{New York}, \bibinfo{year}{1988}.

\end{thebibliography}
\end{document}